\documentclass[10pt,a4paper]{article} 
\usepackage{amsmath,amssymb,amsthm,amsfonts,amscd,euscript,verbatim, t1enc, newlfont, graphicx, cancel}
\usepackage{hyperref}
\usepackage[all]{xy}
\providecommand{\keywords}[1]{\textbf{\textit{Key words and phrases }} #1}
\providecommand{\subjclass}[1]{\textbf{\textit{2010 Mathematics Subject Classification.}} #1}

\theoremstyle{definition}
\newtheorem{theo}{Theorem}[subsection]

\newtheorem{theore}{Theorem}[section]
\newtheorem{pr}[theo]{Proposition}

 \newtheorem{coro}[theo]{Corollary}
  
\theoremstyle{remark}
\newtheorem{rema}[theo]{Remark}

\newtheorem{rrema}[theore]{Remark}
\theoremstyle{definition}
\newtheorem{defi}[theo]{Definition}

\numberwithin{equation}{subsection}

\newcommand\cu{\underline{C}}
\newcommand\du{\underline{D}}
\newcommand\eu{\underline{E}}
\newcommand\au{\underline{A}}

\newcommand\bu{\underline{B}}
\newcommand\hu{\underline{H}}

\newcommand\obj{\operatorname{Obj}}

\newcommand\id{\operatorname{id}}
\DeclareMathOperator\adfu{\operatorname{AddFun}}
\DeclareMathOperator\adfur{\operatorname{Fun}_R}

\DeclareMathOperator\kar{\operatorname{Kar}}
 \DeclareMathOperator\ke{\operatorname{Ker}}
 \DeclareMathOperator\cok{\operatorname{Coker}}
\DeclareMathOperator\imm{\operatorname{Im}}
\DeclareMathOperator\co{\operatorname{Cone}}

\DeclareMathOperator\inli{\varinjlim}

\newcommand\hw{{\underline{Hw}}}
\newcommand\hrt{{\underline{Ht}}}

\newcommand\alz{{\aleph_0}}
\newcommand\alo{{\aleph_1}}

\newcommand\wstu{w^{st}}

\newcommand\spe{\operatorname{Spec}}
\newcommand\modd{\operatorname{Mod}}

\newcommand\q{{\mathbb{Q}}}

\newcommand\p{\mathbb{P}}

\newcommand\z{{\mathbb{Z}}}

 \newcommand\lan{\langle}
\newcommand\ra{\rangle}

\newcommand\al{\alpha}
\newcommand\be{\beta}

\newcommand\ns{\{0\}}
\newcommand\ab{\operatorname{Ab}}

\newcommand\cp{\mathcal{P}}
\newcommand\perpp{{}^{\perp}}
\newcommand\opp{^{op}}

\newcommand\kmp{\mathbb{K}_m(\operatorname{Proj}}

\newcommand\ccu{\underline{\tilde{\mathcal{C}}}} 
\newcommand\wu{\tilde{w}}
\newcommand\hwu{{\underline{H\tilde{w}}}}
\newcommand\mmodd{\operatorname{mod}}

\begin{document}

\title{From weight structures to (orthogonal) $t$-structures and back}
\author{Mikhail V. Bondarko
   \thanks{ 
 The main results of the paper were  obtained under support of the Russian Science Foundation grant no. 16-11-
00200.}}\maketitle
\begin{abstract} 
In this paper we study  $t$-structures that are "closely related" to weight structures (on a triangulated category $\cu$).
 A $t$-structure couple $t=(\cu_{t\le 0},\cu_{t\ge 0})$ is said to be right adjacent to a  weight structure $w=(\cu_{w\le 0}, \cu_{w\ge 0})$ if $\cu_{t\ge 0}=\cu_{w\ge 0}$; if this is the case then $t$ can be uniquely recovered from $w$ and vice versa. We prove that 
 if $\cu$ satisfies the Brown representability property (one may say that this is the case for any "reasonable" triangulated category closed with respect to coproducts) then $t$ that is right adjacent to $w$ exists if and only if $w$ is {\it smashing} (i.e., coproducts respect weight decompositions); in this case the heart $\hrt$ is the category of those functors $\hw\opp\to \ab$ that respect products (here $\hw$ is the heart of $w$). Certainly, the dual to this statement is valid is well, and we discuss its relationship to 
  results of B. Keller and P. Nicolas.

We also prove several generalizations and modifications of this result. In particular, we prove that a right adjacent $t$ exists whenever $w$ is a bounded weight structure on a {\it saturated} $R$-linear category $\cu$ (for a noetherian 
  ring $R$). Moreover,  we obtain  1-to-1 correspondences between bounded weights structures on $\cu$ and the classes of those bounded $t$-structures on it such that $\hrt$ has either enough projectives or injectives whenever $\cu$ equals the  derived category of perfect complexes $D^{perf}(X)$ 
 for $X$ that is regular and  proper over $\spe R$.

Furthermore, we generalize 
 the aforementioned existence statement to construct (under certain assumptions) a $t$-structure $t$ on a triangulated category $\cu'$ 
 that is {\it right orthogonal} to $w$; here  $\cu$ and $\cu'$ are subcategories of a common triangulated category $\du$. In particular, if  $X$ is proper over $\spe R$ but not necessarily regular then one can take $\cu=D^{perf}(X)$, $\cu'=D^b_{coh}(X)$ or $\cu'=D^-_{coh}(X)$, and $\du=D_{qc}(X)$.  
We 
 also study hearts of orthogonal $t$-structures and their restrictions, and prove some statements on "reconstructing" weight structures from orthogonal $t$-structures. 

The main tool of this paper is the notion virtual $t$-truncations of (cohomological) functors; these are defined in terms of weight structures and "behave as if they come from $t$-truncations of representing objects" whether $t$ exists or not.
\end{abstract}
\subjclass{Primary 18E30;  Secondary 18E40, 18F20, 18G05, 18E10, 16E65.}

\keywords{Triangulated category, weight structure, $t$-structure, virtual $t$-truncation, pure functor, coherent sheaves, perfect complexes.}

\tableofcontents
 \section*{Introduction}
This paper is dedicated to the study of those $t$-structures that are "closely related" to weight structures (on various triangulated categories). 

Let us 
 recall a bit of history.  $t$-structures on triangulated categories have become important tools in homological algebra since their introduction in \cite{bbd}. Respectively, their study and construction is an actual and non-trivial question.
 Next, in \cite{konk} and \cite{bws} a rather similar notion of a weight structure $w$ on a triangulated category $\cu$ was introduced.
Moreover, in ibid. a $t$-structure $t=(\cu_{t\le 0}, \cu_{t\ge 0})$  was said to be {\it adjacent} to $w$ if $\cu_{t\ge 0}=\cu_{w\ge 0}$, and certain examples of adjacent structures were constructed.\footnote{In contrast to ibid. and \cite{bger}, in the current paper we use the so-called homological convention for $t$ and $w$, and say that $t$ is right adjacent to $w$ if $\cu_{t\ge 0}=\cu_{w\ge 0}$. Besides, D. Paukstello and several other authors use the term "co-$t$-structure" instead of "weight structure".} 
Furthermore, in \cite{bger} 
 for a $t$-structure $t$ on a triangulated category $\cu'$ that is related to $\cu$ by means of  a {\it duality} bi-functor 
 a more general notion of {\it (left) orthogonality} of  a weight structure $w$ on $\cu$  to $t$ was introduced.\footnote{These definitions are contained in Definition \ref{ddual} below; however, our main examples have motivated us to concentrate on the particular case where $\cu$ and $\cu'$ are subcategories of a common triangulated category $\du$ for most of this paper.}
Also, the relationship between the hearts of adjacent and orthogonal structures was studied in detail. 

Next, if $w$ is (left or right) adjacent to $t$ then it determines $t$ uniquely and vice versa. 
 Yet the only previously existing way of constructing $t$ if $w$ is given was to 
 use certain "nice generators" of $w$ (see Definition \ref{dcomp}(\ref{dgenw}), \S4.5 of \cite{bws}, Theorem 3.2.2 of \cite{bwcp}, and Proposition 5.3.1 of \cite{bpgws}). However, already in \cite{bws} the notion of virtual $t$-truncations for (co)homological functors was introduced, and it was proved that virtual $t$-truncations possess several nice properties. In particular, it was demonstrated that these are closely related to $t$-structures (whence the name) even though they are defined in terms of weight structures only. 

In the current paper we propose a new construction method. We prove that adjacent and orthogonal $t$-structures can be constructed using virtual $t$-truncations whenever certain "Brown representability-type" assumptions on $\cu$ (and $\cu'$) are known.
 Respectively, our results yield the existence of some 
 new families of $t$-structures. 

Let us formulate one of these results. For a triangulated category $\cu$ that is {\it smashing}, i.e., 
 closed with respect to (small) coproducts, and a weight structure $w$ on it we will say that $w$ is {\it smashing} whenever $\cu_{w\ge 0}$ is closed with respect to $\cu$-coproducts (note that $\cu_{w\le 0}$ is $\coprod$-closed automatically).

\begin{theore}[See Theorem \ref{tsmash}(I)]\label{tadjti}
Let $\cu$ be a smashing triangulated category  
 that satisfies the following Brown representability property: any functor $\cu\opp\to \ab$ that respects ($\cu\opp$)-products is representable. 

Then for  a weight structure $w$  on $\cu$ there exists a $t$-structure $t$ right adjacent to it if and only if $w$ is smashing. Moreover, the heart of $t$ (if $t$ exists) is equivalent to the category of all those additive functors $\hw\opp\to \ab$ that respect products.\footnote{Here $\hw$ is the heart of $w$; note also that $G: \hw\opp\to \ab$ respects products whenever it converts $\hw$-coproducts into products of groups.}\end{theore}
Note here that (smashing) triangulated categories satisfying the Brown representability property 
  have recently become very popular in homological algebra and found applications in various areas of mathematics (thanks to the foundational results of A. Neeman and others); in particular, this property holds if either $\cu$ or $\cu\opp$ is compactly generated. 
Moreover, it is easy to construct vast families of smashing weight structures on $\cu$ (at least) if $\cu$ is compactly generated; see Remark \ref{rsmashex}(1) below. The resulting adjacent $t$-structures are {\it cosmashing} (i.e. $\cu_{t\le 0}$ is closed with respect to $\cu$-products); thus there appears to be no way to construct them using previously known methods.

Certainly the dual to Theorem \ref{tadjti} is valid as well. Moreover, if $\cu\opp$ satisfies the dual Brown representability property and $w$ is both cosmashing and smashing then the left adjacent $t$-structure $t$ (i.e., $\cu_{t\le 0}=\cu_{w\le 0}$) restricts to the subcategory of compact objects of $\cu$ as well as to all other "levels of smallness" for objects. Combining this statement with an existence of weight structures Theorem 3.1 of \cite{kellerw} we obtain a statement on $t$-structures extending  Theorem 7.1 of ibid.  

We also prove certain alternative versions of Theorem \ref{tadjti} that can be applied to "quite small" triangulated categories.
 Instead of the Brown representability condition for $\cu$ one can demand it to satisfy the {\it $R$-saturatedness} one instead (see Definition \ref{dsatur}(2) below; this is an  "$R$-linear finite" version of the Brown representability). Then for any {\it bounded} $w$ on $\cu$ there will exist a $t$-structure right adjacent to it. 
According to 
a saturatedness statement from \cite{neesat} 
 this result can be applied to the  derived category $D^{perf}(X)$  
of perfect complexes on a regular scheme 
  that is proper over the spectrum of a Noetherian ring $R$ 
   (see Proposition \ref{pnee1}(2)). In this case we obtain  1-to-1 correspondences between bounded weights structures on $\cu$ and the classes of those bounded $t$-structures on it such that the heart $\hrt$ of $t$ has either enough projectives or injectives; see Remark \ref{rtem1}(\ref{irecw}).
	
	 Moreover, we prove some generalizations of this "$R$-saturated" existence result  (see Proposition \ref{pnee1}(1) and Theorem \ref{tsatur}); they produce orthogonal $t$-structures 
 on bounded and bounded above derived categories of coherent sheaves  over $X$ for $X$ that is not necessarily regular.

 Furthermore, we study the question when a (fixed) $t$-structure $t$ on a triangulated category $\cu'$ is right orthogonal  to a weight structure $w$ on a certain category $\cu$. 
If $\cu$ is a triangulated subcategory of $\cu'$ 
and $w$ that is left orthogonal to $t$ exists then 
 $\hrt$ has enough projectives, and in certain cases this property of $\hrt$ is sufficient for the existence of $w$; see Theorem \ref{twfromt}. Under some other assumptions we prove the existence of a weight structure $w$ (on a triangulated category $\cu$ that is "large enough") that is {\it (strictly)} left orthogonal to $t$ 
 with $\cu$ not being a subcategory of $\cu'$. However, these assumptions on the heart make these result 
 more difficult to apply than the aforementioned "converse" ones, and the methods of their proofs are less interesting.

\begin{rrema}\label{rneetgen}
The author certainly does not claim that the methods of the current paper are the most general among the existing methods for constructing $t$-structures. In particular, if $\cu$ is a {\it well generated} triangulated category (in particular, this is the case if $\cu$ is  compactly generated or possesses a combinatorial Quillen model; see Proposition 6.10 of \cite{rosibr})  
 then the recent Theorem 2.3 of \cite{neetsgen} gives
all those $t$-structures that are generated by sets of objects of $\cu$ in the sense of Definition \ref{dcomp}(\ref{dgent}) below; this statement essentially vastly generalizes the well-known Theorem A.1 of \cite{talosa}.

Now, 
if $\cu$ is well generated then is generated by a set of its objects as its own localizing subcategory  (see Definition \ref{dcomp}(\ref{dlocal})). 
 Thus for any smashing weight structure $w$ on it Proposition 2.3.2(10) of \cite{bwcp} essentially says the following: there exists a set $\cp\subset \cu_{w\ge 0}$ such that the class $\cu_{w\ge 0}$ is the  smallest cocomplete pre-aisle that is "generated" by 
  $\cp$ in the sense of \cite[\S0]{neetsgen} (cf. Discussion 1.15 of ibid.). 
Hence Theorem 2.3 of \cite{neetsgen} says that there exists a $t$-structure on $\cu$ such that $\cu_{t\ge 0}=\cu_{w\ge 0}$ (see Remark \ref{rtst}(4) below). Hence loc. cit. generalizes the existence of $t$ part of our Theorem \ref{tadjti}.

On the other hand, note that neither loc. cit. nor Theorem A.1 of \cite{talosa} says anything on the hearts of $t$-structures. Also,  there appears to be no chance to extend these existence results to $R$-saturated categories (in any way). 
\end{rrema}

Let us  now describe the contents  of the paper. Some more information of this sort may be found in the beginnings of sections. 

In \S\ref{swt} we give some 
  definitions and conventions, and recall some basics on $t$-structures and weight structures.

In \S\ref{sortvtt} we 
 define  virtual $t$-truncations of functors and prove several nice properties for them. We also relate the existence of orthogonal $t$-structures to virtual $t$-truncations; this gives general if and only if criteria for the existence of adjacent $t$-structures. 

In  \S\ref{smashb} we study smashing triangulated categories along with the existence of $t$-structures adjacent to weight structures on them and certain restrictions of these $t$-structures. We also consider weight structures extended from subcategories of compact objects. 

In \S\ref{scoh} we study adjacent and orthogonal $t$-structures on $R$-linear triangulated categories; this includes $R$-saturated categories and various derived categories of (quasi)coherent sheaves.

In \S\ref{sortw} we try to answer the question whether the results of  previous sections give all $t$-structures that are adjacent to weight structures on the corresponding categories. So we study  criteria ensuring the existence of a weight structure
 that is left adjacent or orthogonal  to a given $t$-structure $t$; under certain assumptions we prove that the answer to our question is affirmative. 

The author is deeply grateful to prof. A. Neeman for calling his attention to \cite{kellerw} as well as for writing his extremely interesting texts that are crucial for the current paper, and also to prof. L. Positselski for an online lesson on the coherence of rings. 

\section{A reminder on weight structures and $t$-structures}\label{swt}

In this section we recall the notions of $t$-structures and weight structures, along with orthogonality and adjacency for them. 

In \S\ref{snotata}  we introduce some categorical notation and recall some basics on $t$-structures. 

In \S\ref{sws} we recall some of the theory of weight structures.

In \S\ref{sadj} we recall the definitions of adjacent and orthogonal weight and $t$-structures that are central for this paper.

\subsection{Some categorical and $t$-structure notation}\label{snotata}

\begin{itemize}

\item All products and coproducts in this paper will be small.

\item Given a category $C$ and  $X,Y\in\obj C$  we will write
$C(X,Y)$ for  the set of morphisms from $X$ to $Y$ in $C$.

\item For categories $C',C$ we write $C'\subset C$ if $C'$ is a full 
subcategory of $C$.

\item Given a category $C$ and  $X,Y\in\obj C$, we say that $X$ is a {\it
retract} of $Y$ 
 if $\id_X$ can be 
 factored through $Y$.\footnote{Certainly,  if $C$ is triangulated or abelian, 
then $X$ is a retract of $Y$ if and only if $X$ is its direct summand.}\ 

\item A 
 (not necessarily additive) subcategory $\hu$ of an additive category $C$ 
is said to be {\it retraction-closed} in $C$ if it contains all retracts of its objects in $C$.

\item  For any $(C,\hu)$ as above the full subcategory $\kar_{C}(\hu)$ of 
 $C$ whose objects
are all retracts of 
 (finite) direct sums of objects 
$\hu$ in $C$ will be called the {\it Karoubi-closure} of $\hu$ in $C$; note that this subcategory is obviously additive and retraction-closed in $C$. 

\item  We will say that $C$ is {\it Karoubian} if any idempotent morphism yields a direct sum decomposition in it. 


\item The symbol $\cu$ below will always denote some triangulated category;  it will often be endowed with a weight structure $w$. The symbols $\cu'$ and $\du$ will  also be used  for triangulated categories only.

\item For any  $A,B,C \in \obj\cu$ we will say that $C$ is an {\it extension} of $B$ by $A$ if there exists a distinguished triangle $A \to C \to B \to A[1]$.

\item A class $\cp\subset \obj \cu$ is said to be  {\it extension-closed}
    if it 
		is closed with respect to extensions and contains $0$.  

\item 
We will write $\lan \cp\ra$ for the smallest  full retraction-closed triangulated subcategory of $\cu$ containing $\cp$; we will  call  $\lan \cp\ra$  the triangulated subcategory {\it densely generated} by $\cp$ (in particular, in the case $\cu=\lan \cp \ra$).

Moreover, the smallest  {\bf strict}  full triangulated subcategory of $\cu$ containing $\cp$ will be called the subcategory {\it strongly generated} by $\cp$.

\item The smallest additive retraction-closed extension-closed class of objects of $\cu$ containing   $\cp$  will be called the {\it 
envelope} of $\cp$.  

\item For $X,Y\in \obj \cu$ we will write $X\perp Y$ if $\cu(X,Y)=\ns$. 

For
$D,E\subset \obj \cu$ we write $D\perp E$ if $X\perp Y$ for all $X\in D,\
Y\in E$.

Given $D\subset\obj \cu$ we  will write $D^\perp$ for the class
$$\{Y\in \obj \cu:\ X\perp Y\ \forall X\in D\}.$$
Dually, ${}^\perp{}D$ is the class
$\{Y\in \obj \cu:\ Y\perp X\ \forall X\in D\}$.

\item Given $f\in\cu (X,Y)$, where $X,Y\in\obj\cu$, we will call the third vertex
of (any) distinguished triangle $X\stackrel{f}{\to}Y\to Z$ a {\it cone} of
$f$.\footnote{Recall 
that different choices of cones are connected by non-unique isomorphisms.}\

\item Below $\au$ will always  denote some abelian category; $\bu$ is an additive category.

\item 
All complexes in this paper will be cohomological.

We will write $K(\bu)$ for the homotopy category
of 
 complexes over $\bu$. Its full subcategory of bounded complexes will be denoted by $K^b(\bu)$. We will write $M=(M^i)$ if $M^i$ are the terms of the complex $M$ (in the cohomological indexing).

	\item We will say that an additive covariant (resp. contravariant) functor from $\cu$ into $\au$ is {\it homological} (resp. {\it cohomological}) if it converts distinguished triangles into long exact sequences.
	
	For a (co)homological functor $H$ and $i\in\z$ we will write $H_i$ (resp. $H^i$) for the composition $H\circ [-i]$. 

\end{itemize}

Let us now recall the notion of  a $t$-structure (mainly to fix  notation). 

\begin{defi}\label{dtstr}

A couple of subclasses  $\cu_{t\le 0},\cu_{t\ge 0}\subset\obj \cu$ will be said to be a
$t$-structure $t$ on $\cu$  if 
they satisfy the following conditions:

(i) $\cu_{t\le 0}$ and $\cu_{t\ge 0}$  are strict, i.e., contain all
objects of $\cu$ isomorphic to their elements.

(ii) $\cu_{t\le 0}\subset \cu_{t\le 0}[1]$ and $\cu_{t\ge 0}[1]\subset \cu_{t\ge 0}$.

(iii)  $\cu_{t\ge 0}[1]\perp \cu_{t\le 0}$.

(iv) For any $M\in\obj \cu$ there exists a  {\it $t$-decomposition} distinguished triangle
\begin{equation}\label{tdec}
L_tM\to M\to R_tM{\to} L_tM[1]
\end{equation} such that $L_tM\in \cu_{t\ge 0}, R_tM\in \cu_{t\le 0}[-1]$.

2. $\hrt$ is the full subcategory of $\cu$ whose object class is $\cu_{t=0}=\cu_{t\le 0}\cap \cu_{t\ge 0}$.
\end{defi}

We will also give some auxiliary definitions.

\begin{defi}\label{dtstro}
1. For any $i\in \z$ we will use the notation $\cu_{t\le i}$ (resp. $\cu_{t\ge i}$) for the class $\cu_{t\le 0}[i]$ (resp. $\cu_{t\ge 0}[i]$). 

2. $\hrt$ is the full subcategory of $\cu$ whose object class is $\cu_{t=0}=\cu_{t\le 0}\cap \cu_{t\ge 0}$.

3. We will say that $t$ is {\it left (resp. right) non-degenerate} if $\cap_{i\in \z}\cu_{t\ge i}=\ns$ (resp. $\cap_{i\in \z}\cu_{t\le i}=\ns$). 

Moreover, we will say that $t$ is {\it non-degenerate} if it is both left and right non-degenerate. 

4. We will say that $t$ is {\it bounded below} if $\cup_{i\in \z}\cu_{t\ge i}=\obj\cu$.

Moreover, we will  say that $t$ is {\it bounded} if the equality $\cup_{i\in \z}\cu_{t\le i}=\obj\cu$ is valid as well.

5. Let $\du$ be a full triangulated subcategory of $\cu$.

We will say that $t$ {\it restricts} to $\du$ whenever the couple $t_{\du}= (\cu_{t\le 0}\cap \obj \du,\ \cu_{t\ge 0}\cap \obj \du)$ is a $t$-structure on $\du$.
\end{defi}

 \begin{rema}\label{rtst}
Let us recall some well-known properties of $t$-structures (cf. \S1.3 of \cite{bbd}).

1. 
The triangle (\ref{tdec}) is canonically and functorially determined by $M$. Moreover, $L_t$ is right adjoint to the embedding $ \cu_{t\ge 0}\to \cu$ (if we consider $ \cu_{t\ge 0}$ as a full subcategory of $\cu$) and $R_t$ is left adjoint to  the embedding $ \cu_{t\le -1}\to \cu$; respectively, $L_t$ and $R_t$ are  connected with $\id_{\cu}$ by means of canonical natural transformations. 

For any $n\in \z$ we will use the notation $t_{\ge n}$ for the functor $[-n]\circ L_t\circ [n]$, and $t_{\le n}=[-n-1]\circ L_t\circ [n+1]$.


2. 
$\hrt$ is 
  an abelian category with short exact sequences corresponding to distinguished triangles in $\cu$.

Moreover, have a canonical isomorphism of functors $L_t\circ [1] \circ R_t\circ [-1] \cong [1] \circ R_t\circ [-1] \circ L_t$ (if we consider these functors as endofunctors of $\cu$). This composite functor $H^t$ actually takes values in $\hrt\subset \cu$, and it is homological if considered this way. 

3. 
We have  $\cu_{t\ge 0}=\perpp \cu_{t\le -1}$. 
 Thus $t$ is uniquely determined by  $\cu_{t\le 0}$. 

Moreover, the notion of $t$-structure is self-dual (cf. Proposition \ref{pbw}(\ref{idual}) below). Hence $\cu_{t\le 0}=(\cu_{t\ge 1})^{\perp}$, and thus $t$ is uniquely determined by $\cu_{t\ge 0}$ as well. 

4. Though in \cite{bbd} where $t$-structures were introduced, in the papers of A. Neeman mentioning this notion, 
 and in several preceding papers of the author the "cohomological convention" for $t$-structures was used, in the current text we 
 use the homological convention; the reason for this is that it is coherent with the homological convention for weight structures (see Remark \ref{rstws}(3) below). Respectively, 
 our notation $\cu_{t\ge 0}$ 
 corresponds to the class $\cu^{t\le 0}$ in the cohomological convention. \end{rema}

\subsection{Some basics on weight structures}\label{sws}

Let us recall some basic  definitions of the theory of weight structures. 

\begin{defi}\label{dwstr}

I. A pair of subclasses $\cu_{w\le 0},\cu_{w\ge 0}\subset\obj \cu$ 
will be said to define a weight
structure $w$ on a triangulated category  $\cu$ if 
they  satisfy the following conditions.

(i) $\cu_{w\le 0}$ and $\cu_{w\ge 0}$ are 
retraction-closed in $\cu$ (i.e., contain all $\cu$-retracts of their objects).

(ii) {\bf Semi-invariance with respect to translations.}

$\cu_{w\le 0}\subset \cu_{w\le 0}[1]$, $\cu_{w\ge 0}[1]\subset
\cu_{w\ge 0}$.

(iii) {\bf Orthogonality.}

$\cu_{w\le 0}\perp \cu_{w\ge 0}[1]$.

(iv) {\bf Weight decompositions}.

 For any $M\in\obj \cu$ there
exists a distinguished triangle
$$L_wM\to M\to R_wM {\to} L_wM[1]$$
such that $L_wM\in \cu_{w\le 0} $ and $ R_wM\in \cu_{w\ge 0}[1]$.
\end{defi}

We will also need the following definitions.

\begin{defi}\label{dwso}
Let $i,j\in \z$; assume that a triangulated category $\cu$ is endowed with a weight structure $w$.

\begin{enumerate}
\item\label{idh}
The full category $\hw\subset \cu$ whose objects are
$\cu_{w=0}=\cu_{w\ge 0}\cap \cu_{w\le 0}$ 
 is called the {\it heart} of 
$w$.

\item\label{id=i}
 $\cu_{w\ge i}$ (resp. $\cu_{w\le i}$, resp. $\cu_{w= i}$) will denote the class $\cu_{w\ge 0}[i]$ (resp. $\cu_{w\le 0}[i]$, resp. $\cu_{w= 0}[i]$).

\item\label{id[ij]}
$\cu_{[i,j]}$  denotes $\cu_{w\ge i}\cap \cu_{w\le j}$.\footnote{If $i>j$ and  $M\in \cu_{[i,j]}$ then $M\perp M$ by the orthogonality axiom; thus $\cu_{[i,j]}=\ns$.}

\item\label{idrest}
Let $\du$ be a full triangulated subcategory of $\cu$.

We will say that $w$ {\it restricts} to $\du$ whenever the couple $w_{\du}= (\cu_{w\le 0}\cap \obj \du,\ \cu_{w\ge 0}\cap \obj \du)$ is a weight structure on $\du$.

Moreover, in this case we will also say that $w$ is an {\it extension} of $w_{\du}$.

\item\label{ilrd} We will say that $M$ is left (resp., right) {\it $w$-degenerate} (or {\it weight-degenerate} if the choice of $w$ is clear) if $M$ belongs to $ \cap_{i\in \z}\cu_{w\ge i}$ (resp.    to $\cap_{i\in \z}\cu_{w\le i}$).

\item\label{iwnlrd} We will say that $w$ is left (resp., right) {\it non-degenerate} if all left (resp. right) weight-degenerate objects of $\cu$ are zero.

\item\label{idbob} We will call $\cup_{i\in \z} \cu_{w\ge i}$ (resp. $\cup_{i\in \z} \cu_{w\le i}$) the class of {\it $w$-bounded below} (resp., {\it $w$-bounded above}) objects of $\cu$.

Moreover, we will say that $w$ is {\it bounded below} (resp. {\it bounded above}, resp. {\it bounded}) if all objects of $\cu$ are bounded below (resp.  bounded above, resp. are bounded both below and above).

\item\label{id6}  We will say that a subcategory $\hu\subset \cu$ 
 is {\it negative} (in $\cu$) if $\obj \hu\perp (\cup_{i>0}\obj (\hu[i]))$.
\end{enumerate}
\end{defi}

\begin{rema}\label{rstws}

1. A  simple (and still  useful) example of a weight structure comes from the stupid filtration on the homotopy categories of cohomological complexes $K(\bu)$ for an arbitrary additive  $\bu$ (it can also be restricted to bounded complexes; see Definition \ref{dwso}(\ref{idrest})). In this case $K(\bu)_{\wstu\le 0}$ (resp. $K(\bu)_{\wstu\ge 0}$) is the class of objects that are
homotopy equivalent to complexes  concentrated in degrees $\ge 0$ (resp. $\le 0$); see Remark 1.2.3(1) of \cite{bonspkar} for more detail. 

We will use this notation below. 
 The heart of this weight structure $\wstu$ is the Karoubi-closure  of $\bu$  in  $K(\bu)$; hence it is equivalent to $\kar(\bu)$ (see Remark 2.1.4(2) of \cite{bsnew}).  

2. A weight decomposition (of any $M\in \obj\cu$) is almost never canonical. 

Still for any $m\in \z$ the axiom (iv) gives the existence of a distinguished triangle \begin{equation}\label{ewd} w_{\le m}M\to M\to w_{\ge m+1}M\to (w_{\le m}M)[1] \end{equation}  with some $ w_{\le m}M\in \cu_{w\le m}$   and $ w_{\ge m+1}M\in \cu_{w\ge m+1}$; we will call it an {\it $m$-weight decomposition} of $M$.

 We will often use this notation below (even though $w_{\ge m+1}M$ and $ w_{\le m}M$ are not canonically determined by $M$); we will call any possible choice either of $w_{\ge m+1}M$ or of $ w_{\le m}M$ (for any $m\in \z$) a {\it weight truncation} of $M$. Moreover, when we will write arrows of the type $w_{\le m}M\to M$ or $M\to w_{\ge m+1}M$ we will always assume that they come from some $m$-weight decomposition of $M$.

3. In the current paper we use the ``homological convention'' for weight structures; 
it was previously used in   
\cite{bonspkar}, 
\cite{bpgws},  \cite{bokum},  in \cite{bvt}, and in \cite{bsnew}, 
whereas in \cite{kellerw},  \cite{bws}, and in  \cite{bger}  the ``cohomological convention'' was used. In the latter convention 
the roles of $\cu_{w\le 0}$ and $\cu_{w\ge 0}$ are interchanged, i.e., one
considers   $\cu^{w\le 0}=\cu_{w\ge 0}$ and $\cu^{w\ge 0}=\cu_{w\le 0}$. So,  a
complex $X\in \obj K(\bu)$ whose only non-zero term is the fifth one (i.e.,
$X^5\neq 0$) has weight $-5$ in the homological convention, and has weight $5$
in the cohomological convention. Thus the conventions differ by ``signs of
weights''; 
 $K(\bu)_{[i,j]}$ is the class of retracts of complexes concentrated in degrees
 $[-j,-i]$. 
 
 We also recall that D. Pauksztello has introduced weight structures independently (in \cite{konk}); he called them
co-t-structures. 
\end{rema}


\begin{pr}\label{pbw}
Let  
$m\le n\in\z$, $M,M'\in \obj \cu$, $g\in \cu(M,M')$. 

\begin{enumerate}
\item \label{idual}
The axiomatics of weight structures is self-dual, i.e., for $\cu'=\cu^{op}$
(so $\obj\cu'=\obj\cu$) there exists the (opposite)  weight structure $w'$ for which $\cu'_{w'\le 0}=\cu_{w\ge 0}$ and $\cu'_{w'\ge 0}=\cu_{w\le 0}$.

\item\label{iort}
 $\cu_{w\ge 0}=(\cu_{w\le -1})^{\perp}$ and $\cu_{w\le 0}={}^{\perp} \cu_{w\ge 1}$.

\item\label{icoprod} $\cu_{w\le 0}$ is closed with respect to all (small) coproducts that exist in $\cu$.

\item\label{iext} 
 $\cu_{w\le 0}$, $\cu_{w\ge 0}$, and $\cu_{w=0}$ are additive and extension-closed. 


\item\label{icompl} 
				For any (fixed) $m$-weight decomposition of $M$ and an $n$-weight decomposition of $M'$  (see Remark \ref{rstws}(2))
 $g$ can be extended 
to a 
morphism of the corresponding distinguished triangles:
 \begin{equation}\label{ecompl} \begin{CD} w_{\le m} M@>{c}>>
M@>{}>> w_{\ge m+1}M\\
@VV{h}V@VV{g}V@ VV{j}V \\
w_{\le n} M'@>{}>>
M'@>{}>> w_{\ge n+1}M' \end{CD}
\end{equation}

Moreover, if $m<n$ then this extension is unique (provided that the rows are fixed).

\item\label{isplit} If $A\to B\to C\to A[1]$ is a $\cu$-distinguished triangle 
 and $A,C\in  \cu_{w=0}$ then this distinguished triangle splits; hence $B\cong A\bigoplus C\in \cu_{w=0}$.

\item\label{iwdmod} If $M$ belongs to $ \cu_{w\le 0}$ (resp. to $\cu_{w\ge 0}$) then it is a retract of any choice of $w_{\le 0}M$ (resp. of $w_{\ge 0}M$).

\item\label{iwd0} 
 If $M\in \cu_{w\ge m}$ 
 then $w_{\le n}M\in \cu_{[m,n]}$ (for any $n$-weight decomposition of $M$).

\item\label{igenlm}
The class $\cu_{[m,l]}$ is the 
extension-closure of $\cup_{m\le j\le l}\cu_{w=j}$.

\item\label{iuni} Let  $v$ be another weight structure for $\cu$; assume   $\cu_{w\le 0}\subset \cu_{v\le 0}$ and $\cu_{w\ge 0}\subset \cu_{v\ge 0}$.      
  Then $w=v$ (i.e., these inclusions are equalities).
\end{enumerate}
\end{pr}
\begin{proof}
All of these statements were essentially proved in \cite{bws} (yet pay attention to Remark \ref{rstws}(3) above!). 
\end{proof}

\subsection{On orthogonal  and adjacent  structures}\label{sadj}

Now let us give a certain definition of orthogonality for weight and $t$-structures. 
 Till \S\ref{sdual} we will only consider a particular case of the general notion introduced in \cite{bger} (see Remark \ref{rort}(\ref{irort1}) and Definition \ref{ddual} below).

\begin{defi}\label{dort}
Assume that $\cu$ and $\cu'$ are (full) triangulated subcategories of a triangulated category $\du$, $w$ is a weight structure on $\cu$ and $t$ is a $t$-structure on $\cu'$

1. 
 We will say that $w$ is {\it left orthogonal} (or {\it left $\du$-orthogonal}) to $t$ or that $t$ is {\it right orthogonal} to $w$ whenever $\cu_{w\le 0}\perp_{\du} \cu'_{t\ge 1}$ and $\cu_{w\ge 0}\perp_{\du} \cu'_{t\le -1}$.

2.  Dually,  we will say that $w$ is {\it right orthogonal} (or {\it right $\du$-orthogonal}) to $t$ or that $t$ is {\it left orthogonal} to $w$ whenever $ \cu'_{t\ge 1} \perp_{\du} \cu_{w\le 0}$ and $\cu'_{t\le -1} \perp_{\du} \cu_{w\ge 0} $.

3. If $\cu=\cu'=\du$ and $w$ is left or right orthogonal to $t$ we will also say that $w$ is (left or right) {\it adjacent} to $t$.

4. We will say that $t$ is  {\it strictly right orthogonal} to $w$ and $w$ is {\it strictly left orthogonal} to $t$  if $\cu'_{t\ge 1}=\cu_{w\le 0}^{\perp_{\du}}\cap \obj \cu'$ and $\cu'_{t\le -1}= {}^{\perp_{\du}} \cu_{w\ge 0} \cap \obj \cu'$.

\end{defi}

\begin{rema}\label{rplr}
We will mostly treat the case where $w$ is {\bf left} orthogonal to $t$. Respectively, we will say that $w$ is orthogonal (resp. adjacent) to $t$ of that $t$ is orthogonal to $w$ to mean that $w$ is left orthogonal (resp. adjacent) to $t$ and $t$ is right orthogonal to $w$.
\end{rema}

Let us now relate the latter definition to the notion of adjacent structures introduced in \cite{bws}.

\begin{pr}\label{portadj}
For $\cu$, $w$, and $t$ as in Definition \ref{dort}(3) we have the following: $w$ is (left) adjacent to $t$ if and only if $\cu_{w\ge 0}=\cu_{t\ge 0}$.

\end{pr}
\begin{proof}
If $\cu_{w\ge 0}=\cu_{t\ge 0}$ then $w$ is  (left) adjacent to $t$ immediately from the orthogonality axioms of weight and $t$-structures (see Definition \ref{dtstr}(iii) and Definition \ref{dwstr}(iii)). Conversely, if $w$ is  adjacent to $t$ then combining the orthogonality conditions with Proposition \ref{pbw}(\ref{iort}) and Remark \ref{rtst}(3)  we obtain that
$\cu_{t\ge 0}\subset \cu_{w\ge 0}$ and $\cu_{w\ge 0}\subset \cu_{t\ge 0}$. Hence $\cu_{w\ge 0}= \cu_{t\ge 0}$ as desired.
\end{proof}

\begin{rema}\label{rort}
\begin{enumerate}
\item\label{irort1}
In Definition 2.5.1 of \cite{bger} and Definition 2.3.1 of \cite{bgn}  orthogonality was defined in terms of {\it dualities} of triangulated categories; see Definition \ref{ddual}(1,2) and Remark \ref{rdualzero}(1) below. The reader may easily check that all the statements of this paper that concern orthogonal structures can be generalized to  $\cu$ and $\cu'$  related by an arbitrary duality $\Phi:\cu\opp\times \cu'\to \au$ for an abelian category $\au$; cf. Remark \ref{rprefl} below.

We prefer to avoid dualities in most of this paper due to the reason that we don't have 
 many interesting examples of orthogonal structures in this more general setting.

\item\label{irort2} Proposition \ref{portadj} says that our definition of adjacent "structures" is essentially equivalent to the original Definition  4.4.1 of \cite{bws} (yet cf. Remark \ref{rstws}(3) and note that the definition of left and right adjacent weight and $t$-structures in loc. cit. was "symmetric", i.e., $w$ being left adjacent to $t$ and $t$ being left adjacent to $w$ were synonyms; in contrast, our current convention follows Definition 3.10 of \cite{postov}). 

\item\label{irort3}
Recall also that the notions of weight and $t$-structures essentially have a common generalization; so, both of these yield certain {\it torsion theories}  as defined in \cite{iyayo} (this is the same thing as a {\it complete Hom-orthogonal pair} in the terms of \cite{postov}); see \S3 of \cite{bvt}. Respectively, our definition of adjacent structures is a particular case of Definition 2.2(3) of ibid.

Note also that 
 certain shifts of the classes 
$\cu_{w\le 0}$, $\cu_{w\ge 0}=\cu_{t\ge 0}$, and $\cu_{t\le 0}$  give  a {\it suspended TTF triple} in the sense of   \cite[Definition 2.3]{humavit}. \end{enumerate}
\end{rema}

\section{
On virtual $t$-truncations and their relation to orthogonal $t$-structures}\label{sortvtt} 

This section is dedicated to the virtual $t$-truncations of functors (these come from weight structures) and their general relationship 
with orthogonal $t$-structures.

In \S\ref{svtt} we recall the definition of virtual $t$-truncations of (co)homological functors and study their (easily defined) weight range.

In \S\ref{sgencrit} we introduce the notion of "reflection" (in the context of Definition \ref{dort}) and relate the existence of orthogonal $t$-structures to virtual $t$-truncations.

\subsection{Virtual $t$-truncations and their weight range}\label{svtt} 

We  recall the   notion of virtual $t$-truncations for a cohomological functor $H:\cu\to \au$ (as defined in \S2.5 of \cite{bws} and studied in more detail in \S2 of \cite{bger}). These truncations allow us to "slice" $H$ into $w$-pure pieces (see Remark \ref{rpure}(1--2) below).  

\begin{defi}\label{dvtt}

Assume that $\cu$ is endowed with a weight structure $w$, $ n\in \z$, and $\au$ is  an abelian category.

1. Let $H$ be a cohomological functor from $\cu$ into $\au$. 

 We define the {\it virtual $t$-truncation} functors $\tau_{\le n }(H)$ (resp. $\tau_{\ge n }(H)$)  by the correspondence $$M\mapsto\imm (H(w_{\le n+1}M)\to H(w_{\le n}M)) ;$$ 
(resp. $M\mapsto\imm (H(w_{\ge n}M)\to H(w_{\ge n-1}M)) $); here we take arbitrary choices of  
the corresponding weight truncations of $M$ and connect them using Proposition \ref{pbw}(\ref{icompl}) in the case $g=\id_M$.

2. Let $H':\cu\to \au$ be a homological functor.  Then we will write $\tau_{\le n }(H')$ for the correspondence $M\mapsto\imm (H'(w_{\le n}M)\to  H'(w_{\le n+1}M)) $ and $\tau_{\ge n }(H')=M\mapsto\imm (H'(w_{\ge n-1}M)\to H'(w_{\ge n}M)) $ (here we take the same connecting arrows between weight truncations of $M$ as above).

3.  Assume that $\cu$ is a full triangulated subcategory of a triangulated category $\du$. Then for any $M\in \obj \du$ we will write $H_M=H_{M,\cu}$  (resp. $H^M=H^M_{\cu}$) for the restriction of the functor (co)represented by $M$ to $\cu$ (thus $H_M$ and $H^M$ are functors from $\cu$ into $\ab$).  Moreover, sometimes we will say that these functors are $\du$-Yoneda ones, and that $H_M$ (resp. $H^M$) is $\du$-(co)represented by $M$.
\end{defi}

We recall the main properties of these constructions. 

\begin{pr}\label{pwfil}
In the notation of the previous definition the  following statements are valid.

1. The objects $\tau_{\le n}(H)(M)$ and $\tau_{\ge n}(H)(M)$ are $\cu$-functorial  in $M$  (and essentially do not depend on any choices).

2. The functors $\tau_{\le n}(H)$ and $\tau_{\ge n}(H)$ are cohomological.

3. There exist natural transformations that  yield a long exact sequence 
\begin{equation}\label{evtt}
\begin{gathered} 
\dots \to \tau_{\le n-1}(H)\circ [-1] \to \tau_{\ge n }(H)\to H \\ \to \tau_{\le n -1}(H)\to \tau_{\ge n }(H)\circ [1]\to H^{-1}\to\dots\end{gathered} 
 \end{equation}  (i.e.,  the result of applying this sequence to any object of $\cu$ is a long exact sequence); the shift of this exact sequence by $3$ positions is given by composing the functors with $-[1]$. 

4. Assume that there exists a $t$-structure $t$ that is right orthogonal to $w$ (for certain $\cu'$ and $\du$ as in Definition \ref{dort}). 
Then for any $M\in \obj \cu'$ 
 the functors $ \tau_{\ge n }(H_M)$ and $\tau_{\le n}(H_M)$ (where $H_M$ is defined in Definition \ref{dvtt}(3)) are $\du$-represented (on $\cu$) by $t_{\ge n}M$ and $t_{\le n}M$ (see Remark \ref{rtst}(1)), respectively.

5.  The correspondence  $\tau_{\ge n }(H')$ gives a well-defined homological functor, and there exists a homological analogue of the long exact sequence (\ref{evtt}). 

Moreover, if   there exists a $t$-structure $t$ on a 
 triangulated category $\cu'$ that is  left orthogonal to $w$ (with respect to a triangulated category $\du$ containing $\cu$ and $\cu'$), $\au=\ab$, and the functor $H$ is $\du$-corepresented by an object $N$ of $\cu'$,
 then $\tau_{\ge n }(H')$ is $\du$-corepresented by $t_{\ge n}N$  and $\tau_{\ge n }(H')$ is $\du$-corepresented by $t_{\le n}N$.

6. For any $i\in \z$ we have $\tau_{\le n+i}(H)\cong  \tau_{\le n}(H\circ [i])\circ [-i]$ and $\tau_{\ge n+i}(H)\cong  \tau_{\ge n}(H\circ [i])\circ [-i]$, and also $\tau_{\le n+i}(H')\cong  \tau_{\le n}(H'\circ [i])\circ [-i]$ and $\tau_{\ge n+i}(H')\cong  \tau_{\ge n}(H'\circ [i])\circ [-i]$. 

7. Let $\au'$ be an abelian subcategory of an abelian category $\au$, i.e., $\au'$ is its full subcategory that contains the $\au$-kernel and the $\au$-cokernel of any morphism in $\au'$; 
assume that $w$ restricts to a triangulated subcategory $\cu'$ of $\cu$. 

Then if the restriction of $H$ to $\cu'$ takes it values in $\au'$ then the same is true for all its virtual $t$-truncations.

\end{pr}
\begin{proof} Assertions 1-4 are given by Theorem 2.3.1 of \cite{bger} (yet pay attention to Remark \ref{rstws}(3); one should also invoke Remark \ref{rdualzero}(1) below to obtain assertion 4). Assertion 5  is easily seen to be dual to the previous ones, whereas assertions 6 and 7 follow from our definitions immediately.
\end{proof}


Now we define weight range and relate it to virtual $t$-truncations; some of these statements will be applied elsewhere (only). 

\begin{defi}\label{drange}
 Let $m,n\in \z$; let  $H$ be as above. 

Then we will say that $H$ is {\it of weight range} $\ge m$ (resp. $\le n$, resp.  $[m,n]$)  if it annihilates $\cu_{w\le m-1}$ (resp. $\cu_{w\ge n+1}$, resp. both of these classes). 
\end{defi}

\begin{pr}\label{pwrange}

\begin{enumerate}
\item\label{iwrort} Adopt the notation and assumptions of Definition \ref{dort}(1). Then for 
 $N\in \cu'_{t\le 0}$ (resp.  $N\in \cu'_{t\ge 0}$, resp. $N\in \cu'_{t= 0}$) the corresponding $\du$-Yoneda functor $H_N:\cu\opp\to \ab$ (see Definition \ref{dvtt}(3))  is of weight range $\le 0$ (resp. $\ge 0$, resp. $[0,0]$).

\item\label{iwrvt} 
For $H$ as in Definition \ref{dvtt}(1) the functor $\tau_{\le n}(H)$ is of weight range $\le n$, and $\tau_{\ge m}(H)$ is of weight range $\ge m$.

\item\label{iwridemp} We have $\tau_{\le n}(H)\cong H$ (resp. $\tau_{\ge m}(H)\cong H$) if and only if $H$ is of weight range $\le n$ (resp. of weight range $\ge m$).

\item\label{iwrcomm} We have $\tau_{\le n}(\tau_{\ge m})(H)\cong \tau_{\ge m}(\tau_{\le n})(H)$.

\item\label{iwruni} If a cohomological functor $H'$ from $\cu$ into $\au$ is of weight range $\ge n$ (resp. of weight range $\le n-1$) then any transformation $T:H'\to H$ (resp. $H\to H'$) factors through the transformation $\tau_{\ge n }(H)\to H$ (resp.  $H \to \tau_{\le n -1}(H)$) provided by the formula (\ref{evtt}).

\item\label{iwrpure} The (not necessarily locally small) category of weight range $[0,0]$ cohomological functors from $\cu$ into $\au$  is equivalent to $\adfu(\hw\opp,\au)$  in the obvious natural way. 

\item\label{iwrpureb} Assume that $H$ is a  weight range $[0,0]$ cohomological functor from $\cu$, and $M$ is a bounded above (resp. below) object of $\cu$. Then $H^i(M)=0$ for $i\gg 0$ (resp. for $i\ll 0$).

\item\label{iwfil4} If $H$ is 
of weight range $\ge m$  then  $\tau_{\le n}(H)$ is 
  of weight range $[m,n]$.

Dually, if $H$ is 
of weight range $\le n$  then  $\tau_{\ge m}(H)$ is 
  of weight range $[m,n]$.

\item\label{iwrvan} 
Assume that $m>n$. Then the only functors of weight range  $[m,n]$ are zero ones; thus if $H$ is of weight range $\le n$ (resp. $\ge m$) then $\tau_{\ge m}(H)=0$ (resp. $\tau_{\le n}(H)=0$).

\item\label{iwrperp} 
If a cohomological functor $H'$ (resp. $H''$) from $\cu$ into $\au$ is of weight range $\ge n$ (resp. of weight range $\le n-1$)
then there are no non-zero transformations from $H'$ into $H''$.

\item\label{iwfil3} The (representable) functor  $H_M=\cu(-,M):\cu\opp\to \ab$  if of weight range $\ge m$ if and only if $M\in \cu_{w\ge m}$.

\item\label{iwfil5} If $H$ is of weight range $\ge m$ (resp. $\le m$) then  the morphism  $H(w_{\ge m}M)\to H(M)$ is epimorphic (resp. the morphism  $H(M)\to H(w_{\le m}M)$ is monomorphic); here we take arbitrary choices of the corresponding weight decompositions of $M$ and apply $H$ to the connecting morphisms.

\end{enumerate}
\end{pr}
\begin{proof}

\ref{iwrort}. For $N\in \cu'_{t\le 0}$ and $N\in \cu'_{t\ge 0}$ 
the weight range estimates for the functor $H_N$ prescribed by the assertion 
 are given by the definition of orthogonality, and to obtain the claim for $N\in \cu'_{t= 0}$ one should combine the first two weight range statements.

\ref{iwrvt}. Let $M\in \cu_{w\ge n+1}$. Then we can take $w_{\le n}(M)=0$. Thus $\tau_{\le n}(H)(M)=0$, and we obtain the first part of the assertion. It second part is easily seen to be dual to the first part. 

\ref{iwridemp}. 
This  is precisely Theorem 2.3.1(III.2,3) of \cite{bger} (up to change of notation); assertion \ref{iwrcomm} is given by part II.3 of that theorem.

 \ref{iwruni}. The two statements in the assertion 
 are easily seen to be dual to each other; hence it suffices to consider the case where $H'$ is of weight range $\ge n$. Next, the obvious functoriality of the definition of virtual $t$-truncations gives the
following commutative square of transformations:

\begin{equation}\label{euni}
\begin{CD} 
\tau_{\ge n}H'@>{\tau_{\ge n}T}>> \tau_{\ge n}H \\
@VV{i'}V@VV{i}V \\
H'@>{T}>>H
\end{CD}\end{equation}
(cf. (\ref{evtt})).

Applying assertion \ref{iwridemp} we obtain that the transformation $i'$ is an equivalence. Hence the transformation 
$\tau_{\ge n}T$ yields the factorization in question. 

\ref{iwfil4}. 
Let  $H$ be of weight range $\ge m$.  Then $\tau_{\le n}(H)\cong \tau_{\le n}(\tau_{\ge m})(H)\cong \tau_{\ge m}(\tau_{\le n})(H)$ (according to the two previous assertions). It remains to apply assertion \ref{iwrvt} to obtain the first statement in the assertion,  
whereas its second part is easily seen to be the dual of the first part (and certainly  can be proved similarly).

\ref{iwrvan}. 
For any $l\in \z$ and any cohomological $H$ any choice of an $l$-weight decomposition triangle (cf. (\ref{ewd})) for $M$ gives a long exact sequence
\begin{equation}\label{eles}
\begin{gathered}
\dots \to H((w_{\le l}M)[1])\to  H(w_{\ge l+1}M)\to H(M)\\
\to H(w_{\le l}M)\to H((w_{\ge l+1}M)[-1])\to\dots
\end{gathered}
\end{equation}
The exactness of this sequence in $H(M)$ for $l=n$ immediately gives   the first part of the assertion. 
 Next, the second part is straightforward from the first one combined with assertion \ref{iwfil4}.

\ref{iwrpure}. Immediate from Theorem 2.1.2(2) of \cite{bwcp}.  

\ref{iwrpureb}. Straightforward from the definition of weight range.

\ref{iwrperp}. According to assertion \ref{iwruni}, any transformation in question factors through $\tau_{\ge n }(H'')$; thus it is zero according to assertion \ref{iwrvan}. 

Assertion \ref{iwfil3} is immediate from  Proposition \ref{pbw}(\ref{iort}). 

Assertion \ref{iwfil5} is a straightforward consequence of assertion \ref{iwfil3}; just apply (\ref{eles}) for $l=m$ and for $l=m-1$, respectively. 
\end{proof}

\begin{rema}\label{rpure}
1. Roughly, the statements above say that virtual $t$-truncations of functors behave as if they corresponded to $t$-truncations of objects in a certain triangulated "category of functors" (whence the name; certainly, another justification of this idea is provided by the existence of orthogonal $t$-structures statements that will be proved below). In particular, one can "slice" any functor of weight range $[m,n]$ for $m\le n$ into "pieces" of weight $[i,i]$ for $m\le i\le n$. Now, composing a "slice" of weight range $[i,i]$ with $[i]$ one obtains a functor of  weight range $[0,0]$. 

2. So we recall that functors of this type were studied in detail in (\S2.1 of) \cite{bwcp}; they were called {\it pure} ones due to the relation to Deligne's purity (cf. Remark 2.1.3(3--4) of ibid.).

3. The author suspects that the connecting transformation in part \ref{iwruni} of our proposition is actually unique. 
\end{rema}


We also formulate a simple statement for the purpose of applying it in \cite{bsnew}.

\begin{pr}\label{pcrivtt}
For $M\in \obj \cu$ the following conditions are equivalent.

(i) $M\in \cu_{w\ge 0}$. 

(ii) $H(M)=0$ for any cohomological functor $H$ from $\cu$ into (an abelian category) $\au$ that is of weight range $\le -1$.

(iii) $(\tau_{\le-  1}H_N)(M) =\ns$ for any $N\in \obj \cu$.

(iv) $(\tau_{\le -1}H_M)(M) =\ns$.

\end{pr}
\begin{proof}
Condition (i) implies condition  (ii) by definition; certainly, (iii)   $\implies$ (iv). Next, condition (ii) implies condition (iii) according to Proposition  \ref{pwrange}(\ref{iwrvt}). 

Lastly, if $(\tau_{\le -1}H_M)(M) =\ns$ then the long exact sequence (\ref{evtt}) yields that $(\tau_{\ge -0}H_M)(M)$ surjects onto $\cu(M,M)$. Hence the morphism $\id_M$ factors through $w_{\ge  0}M$; thus $M$ belongs to $ \cu_{w\ge 0}$.
\end{proof}

\subsection{General criteria for the existence of adjacent and orthogonal $t$-structures}\label{sgencrit}

To generalize the criteria below from the "main" case of adjacent structures to certain orthogonal structures we will need the following definition.

\begin{defi}\label{drefl}
Let $\cu$ and $\cu'$ be (full) triangulated subcategories of a triangulated category $\du$ (cf. Definition \ref{dort}).


Then 
 we will say that  $\cu$  {\it reflects  $\cu'$} 
 (in the category $\du$)  whenever the 
$\du$-Yoneda  functor $\cu'\to \adfu(\cu\opp,\ab):M\mapsto H_M$ (see Definition \ref{dvtt}(3))   is fully faithful. 
\end{defi}

Let us relate this notion to orthogonal structures and their properties. The reader may note that some of these proofs can be substantially simplified in the cases $\cu=\cu'$ (and so, for adjacent structures) and $\cu'\subset \cu$.

\begin{pr}\label{prefl}
Adopt the notation of the previous definition.

\begin{enumerate}
\item\label{irefl1} If $\cu'\subset \cu$ then $\cu$ reflects $\cu'$. 

\item\label{irefl3} If $\cu$ reflects $\cu'$ 
  and $w$  is a weight structure on $\cu$ then for the classes $C'_1= \cu_{w\ge 1}^{\perp_{\du}}\cap \obj \cu'$  and $C'_2= \cu_{w\le -1}^{\perp_{\du}}\cap \obj \cu'$ we have $C'_2[1]\perp C'_1$.

\item\label{irefl4} If $\cu$ reflects $\cu'$ 
 and there exists a $t$-structure $t=(\cu'_{t\le 0},\cu'_{t\ge 0})$ on $\cu'$ that is right orthogonal to $w$ then $t$ is also strictly right orthogonal to $w$, i.e., $t=(C'_1,C'_2)$.  Moreover, the corresponding Yoneda-type functor $\hrt\to \adfu(\hw\opp,\ab)$ is fully faithful. 
\end{enumerate}
\end{pr}
\begin{proof}
Assertion  \ref{irefl1}  immediately follows from the Yoneda lemma.

\ref{irefl3}. If $M_1$ belongs to $ \cu_{w\ge 1} ^{\perp_{\du}}$ 
 and $M_2\in  \cu_{w\le 0}^{\perp_{\du}}$ then the 
  $\du$-Yoneda functor $H_{M_1}:\cu\opp\to \ab$ is  of weight range $\le 0$ and $H_{M_2}$  is of weight range $\ge 1$ (see the obvious Proposition \ref{pwrange}(\ref{iwrort})). 
 By Proposition \ref{pwrange}(\ref{iwrperp}) we obtain that there are no non-zero transformations between $H_{M_2}$ and $H_{M_1}$.

Next we assume in addition that $M_1$ and $M_2$ are objects of $\cu'$. Since $\cu$ reflects $\cu'$ this vanishing of transformations statements implies that $M_2\perp M_1$; thus $C'_2[1]\perp C'_1$ indeed.   

\ref{irefl4}. Assume that $t$ is   
 orthogonal to $w$. Then the definition of orthogonality says that  $\cu'_{t\le 0}\subset C'_1$ and  $\cu'_{t\ge 0}\subset C'_2$. 
On the other hand, recall that $\cu'_{t\ge 0}=^{\perp_{\cu'}} \cu'_{t\le -1}$ and $\cu_{t\le 0}=(\cu'_{t\ge 1})^{\perp_{\cu'}} $ (see Remark \ref{rtst}(3)). Since $C'_2\perp C'_1[1]$, we obtain that the converse inclusions are valid as well; thus $t$ is strictly right orthogonal to $w$. 

Next, if $M\in \cu'_{t=0}$ then the definition of orthogonality implies that the  functor $H_M:\cu\opp\to \ab$ (see Definition \ref{dvtt}(3)) is of weight range $[0,0]$. Thus it suffices to recall that $\cu$ reflects $\cu'$ 
 and apply Proposition \ref{pwrange}(\ref{iwrpure}).
\end{proof}

\begin{rema}\label{rhop}
 Our proposition implies that $w$ determines a $t$-structure that is (right) adjacent to it uniquely, and this $t$-structure is strictly right orthogonal to $w$. 
 Certainly, $t$ determines a weight structure that is left adjacent to it uniquely as well; see Proposition \ref{portadj} and Proposition \ref{pbw}(\ref{iort}). Moreover, this argument easily extends to arbitrary {\it Hom-orthogonal pairs} (see Definition 3.1 of \cite{postov}); in particular, it works for {\it torsion theories} (see \S2 of \cite{bvt} and Definition 2.2 of  \cite{iyayo}).

Moreover, the author conjectures that this uniqueness statements are valid for orthogonal torsion theories whenever $\cu$ reflects $\cu'$. 
\end{rema}

\begin{pr}\label{padjpr}
 Assume that $\cu$ reflects $\cu'$ 
(see Definition \ref{drefl}; here we adopt its notation),  $w$ is a weight structure on $\cu$,  $M\in \obj \cu'$, and 
 that for the functor $H_M:\cu\opp\to \ab$  (see Definition \ref{dvtt}(3))  its virtual $t$-truncation $\tau_{\ge 0}H_M$ is represented by some object $M_{\ge 0}$ of $\cu'$.

Then the following statements are valid.

\begin{enumerate}
\item\label{ile1} $M_{\ge 0}$ belongs to $\cu_{w\le -1}^{\perp_{\du}}\cap \obj \cu'$. 

\item\label{ile2}
The natural transformation $\tau_{\ge 0 }(H_M)\to H_M$ mentioned in (\ref{evtt}) is induced by some $f\in \cu'(M_{\ge 0},M)$. 

\item\label{ile3} The object $M_{\le -1}=\co(f)$ belongs to $\cu_{w\ge 0}^{\perp_{\du}}\cap \obj \cu'$. 

\item\label{ile5} 
For $M'\in \obj \cu'$ the $\du$-representability of the  functor  $\tau_{\ge 0}H_{M'}$ by an object of $\cu'$ is equivalent to that of  $\tau_{\le - 1}H_{M'}$.
\end{enumerate}

\end{pr}
\begin{proof}
1.  It suffices to recall that $\tau_{\ge 0}H_M$ is of weight range $\ge 0$ according to  Proposition \ref{pwrange}(\ref{iwrvt})). 

2. This transformation 
 lifts to $\cu'$ since $\cu$ reflects $\cu'$. 

3. For any $N\in \obj \cu$ applying the functor $H^N=\du(N,-)$ to the distinguished triangle $M_{\ge 0}\to M \to M_{\le -1} \to M_{\ge 0}[1]$ one obtains a long exact sequence that yields the following short one: \begin{equation} \label{eshort}
\begin{gathered} 0\to \cok(H^N(M_{\ge 0})\stackrel{h^1_N}{\to} H^N(M))\to H^N(M_{\le -1})\\
 \to \ke (H^N(M_{\ge 0}[1])\stackrel{h^2_N}{\to} H^N(M[1]))\to 0
\end{gathered}
\end{equation} 
So, for any $N\in \cu_{w\ge  0}$ we should check that the homomorphism $h^1_N$ is surjective and $h^2_N$ is injective.

Applying  (\ref{evtt}) to the functor $H_M$ (in the case $n=0$) 
we obtain a long exact sequence of functors 
\begin{equation} \label{evttp} 
\dots\to \tau_{\ge 0 }(H_M)\to H_M\to \tau_{\le -1}(H_M)\to \tau_{\ge 0}(H_M)\circ [1]\to H_M\to \dots
\end{equation}
 Applying this sequence of functors to $N$ we obtain that the surjectivity of  $h^1_N$ along with the injectivity of $h^2_N$ is equivalent to  $\tau_{\le -1}(H_M)(N)=\ns$. 
 Recalling that $\tau_{\le -1}(H_M)$ is of weight range $\le -1$ according to Proposition \ref{pwrange}(\ref{iwrvt}), we conclude the proof.

4. If  $\tau_{\ge 0}H_{M'}$ is representable then the previous assertions imply the existence of a $\du$-distinguished triangle $M'_{\ge 0}\to M' \to M'_{\le -1} \to M'_{\ge 0}[1]$ with $M'_{\le -1}\in \cu_{w\ge 0}^{\perp_{\du}}\cap \obj\cu'$. Then the object $M'_{\le -1} $ $\du$-represents the functor  $\tau_{\le -1}H_{M'}$ according to Theorem 2.3.1(III.4) of \cite{bger} (and so, $\tau_{\le -1}H_{M'}$ is representable). The proof of the converse implication is similar.
\end{proof}

\begin{rema}\label{rprefl}
1. As we have already noted in Remark \ref{rort}(\ref{irort1}), it is not really necessary to assume that $\cu$ and $\cu'$ lie in some common triangulated category $\du$. However, the author does not now of any examples such that no $\du$ exists but $\cu$ reflects $\cu'$ in the easily defined generalized sense of this notion (cf. Definition \ref{ddual} below).

On the other hand, the main statements needed for the construction of orthogonal $t$-structures below are Proposition \ref{prefl}(\ref{irefl3}) and Proposition \ref{padjpr}(\ref{ile2}). The author does not know of any "abstract" conditions on the categories $\cu$ and $\cu'$, and the duality $\Phi:\cu\opp\times \cu'\to \au$ 
 that would allow to generalize our current 
  proofs of these statements. However, it appears that 
 if $\cu$  possesses a "model" that allows to define a triangulated derived category $\cu'$ of "reasonable" functors $\cu\to \au$ (this is certainly the case when $\cu$ is a triangulated subcategory of the homotopy category of a stable model category) then the (corresponding version of) Proposition \ref{padjpr}(\ref{ile2}) is fulfilled. Next, one can probably lift the square (\ref{euni}) to $\cu'$ to obtain eventually that the corresponding version of  Proposition \ref{prefl}(\ref{irefl3}) is valid as well.

Possibly, the author will study this matter in a succeeding paper. 

2. The author does not know whether it makes much sense to take $\au\neq \ab$ in the argument that we have just sketched. Note however that we could have considered a duality with values in $\au=R-\modd$ throughout section \ref{scoh} below.
\end{rema}

Now we are able to prove our main abstract criterion on the existence of an orthogonal $t$-structure.

\begin{theo}\label{trefl}
Assume that 
a triangulated subcategory $\cu$ of $\du$ reflects $\cu'\subset \du$ 
 (see Definition \ref{drefl}); let $w$ be a weight structure on $\cu$. 

 Then the following conditions are equivalent.

(i). There exists a $t$-structure $t$ on $\cu'$ right orthogonal to $w$. 
 
(ii). The functor  $\tau_{\ge 0}H_{M'}$ is $\du$-representable by an object of $\cu'$ for any object $M'$ of $\cu'$. 

(iii). The functor  $\tau_{\le -1}H_{M'}$ is $\du$-representable by an object of $\cu'$ for any object $M'$ of $\cu'$. 

(iv).  For any object $M'$ of $\cu'$ and $i\in \z$ the functors  $\tau_{\ge i}H_{M'}$ and $\tau_{\le i}H_{M'}$ are $\du$-representable by objects of $\cu'$.
\end{theo}
\begin{proof}
  Condition (i) implies conditions (ii) and (iii) according to Proposition \ref{pwfil}(4). Next, conditions (ii) and (iii) are equivalent by Proposition \ref{padjpr}(\ref{ile5}). Moreover, conditions (iv) certainly implies conditions (ii) and (iii), whereas the reverse implication easily follows from Proposition \ref{pwfil}(6).

It remains to prove that condition (ii) implies (i). So,  we should check that the couple $(C'_1,C'_2)$, where $C'_1= \cu_{w\ge 1}^{\perp_{\du}}\cap \obj \cu'$  and $C'_2=\cu_{w\le -1}^{\perp_{\du}}\cap \obj \cu'$ (cf. Proposition \ref{prefl}(\ref{irefl3},\ref{irefl4})) is a $t$-structure.  Now, these classes are certainly closed with respect to $\cu'$-isomorphisms, and the "shift" axiom (ii) of Definition \ref{dtstr} is obviously fulfilled as well. Next, the orthogonality axiom (iii) is given by Proposition \ref{prefl}(\ref{irefl3}).

Hence it remains to check the existence of $t$-decompositions (this is axiom (iv) of $t$-structures), which 
 immediately follows from Proposition \ref{padjpr}(\ref{ile1}--\ref{ile3}).
\end{proof}

Let us also prove (one more) corollary from Proposition \ref{prefl}.  We will apply it in \S\ref{satur} for $\cu'=\cu$; note however that in this case one can avoid using Proposition \ref{prefl}.

\begin{pr}\label{psingen}
Assume that $\cu$ is densely generated by a single object $G$ and $w$ is a bounded weight structure on $\cu$.

1. Then there exists $N\in\z$ such that the class $\cu_{w\le 0}$ is contained in the envelope (see \S\ref{snotata}) of $\{G[j]:\ j<N\}$ and $\cu_{w\ge 0}$ 
 lies in the envelope of $\{G[j]:\ j>-N\}$.

2. Assume that $\cu(G,G[j])=\ns$ for $j\ll 0$, $\cu'\subset \cu$, and there exists a $t$-structure  on $\cu'$ that is right orthogonal to $w$. Then $\cu'_{t\ge 0}= \cu_{w\ge 0}^{\perp_{\du}}\cap \obj \cu'$ and there exists $N'\in\z$ such that $\cu_{t\le 0}\supset \cu_{w\le -N'}\cap \obj \cu'$; 
 hence $t$ is bounded as well.
	
\end{pr}
\begin{proof}
1. Since $w$ is bounded, applying Proposition \ref{pbw}(\ref{igenlm}) we reduce our assertion to the existence of $N$ such that $\cu_{w=0}$ lies in  the envelope of $\{G[j]:\ -N<j<N\}$. Now, Remark 2.3.5(2) of \cite{bwcp} says that there exists a finite set of $G^i\in \cu_{w=0}$\footnote{These objects are the terms of a bounded choice of a {\it weight complex} $t(G)$ of $G$.} such that any element of  $\cu_{w=0}$ is a retract of a direct sum of a (finite) collection of $G^i$. Since the set $\{G\}$ densely generates $\cu$, it remains to choose $N$ such that all of these $G^i$ belong to the envelope of $\{G[j]:\ -N<j<N\}$.

2. Combining Proposition \ref{prefl} with Proposition \ref{pbw}(\ref{iort}) we obtain that $\cu'_{t\ge 0}= \cu_{w\ge 0}^{\perp_{\du}}\cap \obj \cu'$ indeed. 
Hence if $\cu_{t\le 0}\supset \cu_{w\le -N'}\cap \obj \cu'$ then $t$ is bounded since $w$ is.

According to Remark \ref{rtst}(3), to verify the inclusion in question it suffices to check that $\cu_{w\le -N'}\perp \cu'_{t\ge 1}\subset \cu_{w\ge 1}$. 
Now, the existence of $N'$ satisfying this condition is straightforward from our assumption on $G$ along with assertion 1.
\end{proof}

\section{On $t$-structures orthogonal to smashing weight structures}\label{smashb} 


In this section we study the existence of adjacent weight and $t$-structures in triangulated categories closed with respect to (co)products (these are called smashing and cosmashing ones).

In \S\ref{smash} we consider 
 smashing weight structures (on smashing triangulated categories); these are the ones that respect coproducts.

In \S\ref{smashort}  we recall the notion of (dual) Brown representability for smashing triangulated categories, and 
prove Theorem \ref{tadjti}, i.e., that a weight structure $w$ on a category satisfying this condition is left adjacent to a $t$-structure if and only if $w$ is smashing. 
 Certainly, the dual to this statement is also valid; moreover, if $w$ is both cosmashing and smashing then the left adjacent $t$-structure $t$ restricts to the subcategory of compact objects of $\cu$ as well as to all other "levels of smallness" for objects. Combining this statement with an existence of weight structures theorem from \cite{kellerw} we obtain a statement on $t$-structures extending yet another result of ibid.  

In \S\ref{sort}  we study extensions of weight structures from subcategories of compact objects and the corresponding adjacent $t$-structures. 

\subsection{On smashing weight structures}\label{smash} 

We will need a few definitions.

\begin{defi}\label{dsmash}
\begin{enumerate}
\item\label{ismcat}
 We will say that a triangulated category $\cu$  is {\it (co)smashing} if it is closed with respect to (small) 
 coproducts (resp., products).

\item\label{ismw} We will say that a weight structure $w$ on $\cu$ is (co)smashing if $\cu$ is (co)smashing and the class $\cu_{w\ge 0}$ is closed with respect to $\cu$-coproducts (resp., $\cu_{w\le 0}$ is closed with respect to $\cu$-products; cf. Proposition \ref{pbw}(\ref{icoprod})). 

\item\label{ismt} We will say that a $t$-structure $t$ on $\cu$ is (co)smashing if $\cu$ is (co)smashing and the class $\cu_{t\le 0}$ is closed with respect to $\cu$-coproducts (resp., $\cu_{t\ge 0}$ is closed with respect to $\cu$-products; cf.  Remark \ref{rtst}(3)).

\item\label{idcc} 
It will be convenient for us to use the following somewhat clumsy  terminology: 
a cohomological functor  $H'$ from $\cu$ into $\au$ will be called a {\it cp} functor if it converts all (small) coproducts into $\au$-products.

\item\label{idal} 
For an infinite cardinal $\al$ 
 a homological functor $H:\cu\to \au$ 
 is said to be {\it $\al$-small} if for any family $N_i$, $i\in I$, we have
$H(\coprod_{i\in I} N_i)=\inli_{J\subset I,\ \#J<\al}H(\coprod_{j\in J}N_j)$ (i.e., the obvious morphisms $H(\coprod_{j\in J}N_j)\to H(\coprod N_i)$ form a colimit diagram; note that this colimit is filtered). 
\end{enumerate}
\end{defi}

Let us now prove some properties of these notions and relate them to $t$-truncations. 

\begin{pr}\label{psmash}
Assume that $w$ is a smashing weight structure on $\cu$, $H:\cu\to \au$ is a homological functor (where $\au$ is an abelian category), 
$i\in \z$,  and $\al$ is an infinite cardinal. 
  Then the following statements are valid.

\begin{enumerate}
\item\label{ismash1}
If $\al'\ge \al$ 
  then any $\al$-small functor is also $\al'$-small.

\item\label{ismash2} $H$ is $\alz$-small if and only if it respects coproducts.

\item\label{icoprh} The class $\cu_{w=0}$ is closed with respect to $\cu$-coproducts.

\item\label{icopr2} 
Coproducts of $w$-decompositions are weight decompositions as well.

\item\label{icopr6p} Assume that $\au$ is an AB4* 
  category and 
a cohomological functor $H'$ from $\cu$ into $ \au$ is a cp one. Then   $\tau_{\ge  i}(H')$  and $\tau_{\le i}(H')$  are cp functors as well.

\item\label{icopr6n} Assume that $\au$ is an AB5 
  category and $H$ is an $\al$-small functor.  
	 Then the functors $\tau_{\ge  i} (H)$  and $\tau_{\le  i}(H)$  are $\al$-small as well.
\end{enumerate} 
\end{pr}
\begin{proof}
\ref{ismash1}. Assume that 
 $H$ is an $\al$-small functor; fix  an index set $I$ and certain $N_i\in \obj \cu$. Then for any $J\subset I$ we have 
$H(\coprod_{j\in J} N_j)=\inli_{J'\subset J,\ \#J'<\al}H(\coprod_{j'\in J'}N_{j'})$. Combining these statements for all $J\subset I$ (actually, it suffices to take  $J=I$ along with $J$ of cardinality less than $\al'$ only here) one easily obtains that $H(\coprod_{i\in I} N_i)=\inli_{J\subset I,\ \#J<\al'}H(\coprod_{j\in J}N_j)$. 

\ref{ismash2}. 
Since $H$ is additive,   $H(\coprod N_i)=\inli_{J\subset I,\ \#J<\alz}H(\coprod_{j\in J}N_j)$ if and only if $H$  respects coproducts (since this colimit will not change if one will consider only those $J$ that consist of a single element only). 

\ref{icoprh}. This is an easy consequence of  Proposition \ref{pbw}(\ref{iort}); see Proposition 2.3.2(1) of \cite{bwcp}. 

\ref{icopr2}. This is an easy consequence of  Proposition \ref{pbw}(\ref{icoprod}) along with Remark 1.2.2 of \cite{neebook}; it is given by Proposition 2.3.2(3) of \cite{bwcp}.  

\ref{icopr6p}. 
According to Proposition \ref{pwfil}(6), it suffices to verify that the functors  $\tau_{\ge  0}(H')$  and $\tau_{\le -1}(H')$ are cp ones for any cp functor $H'$. 
For 
a family $\{M_i\}$ of objects of $\cu$ we choose certain $-1$ and $-2$-weight decompositions for all $M_i$ (see Remark \ref{rstws}(2)).
According to the previous assertion,  their coproducts give a $-1$ and a $-2$-weight decomposition of $\coprod M_i$, respectively. Moreover, one can certainly obtain the unique morphisms $w_{\le -2}(\coprod M_i)\to w_{\le -1}(\coprod M_i)$  and $w_{\ge -1}(\coprod M_i)\to w_{\ge 0}(\coprod M_i)$ compatible with these decomposition triangles (see Proposition \ref{pbw}(\ref{icompl}) and Definition \ref{dvtt}(1))  as the coproducts of the corresponding morphisms for $M_i$.  Applying our assumptions on $H'$ and $\au$ we obtain that  $\tau_{\ge  0}(H')(\coprod M_i)\cong \prod \tau_{\ge  0}(H')(M_i)$ and 
 $\tau_{\le  -1}(H')(\coprod M_i)\cong \prod \tau_{\ge  -1}(H')(M_i)$.

Similarly, to prove assertion \ref{icopr6n} it suffices  
 to verify that  the functors $\tau_{\ge  0}(H)$  and $\tau_{\le -1}(H)$ are $\al$-small whenever $H$ is. 
One takes the same weight decompositions along with their coproducts corresponding to all subsets $J$ of $I$ of cardinality less than $\al$. Since the colimits in question are filtered ones,  the AB5 assumption on $\au$ allows to compute 
$\inli_{J\subset I,\ \#J<\al}\imm(H(\coprod_{j\in J}w_{\le -2}N_j)\to H(\coprod_{j\in J}w_{\le -1}N_j))$ and 
$\inli_{J\subset I,\ \#J<\al}\imm(H(\coprod_{j\in J}w_{\ge -1}N_j)\to H(\coprod_{j\in J}w_{\ge -0}N_j))$
as the corresponding images of colimits to obtain the statement in question.
\end{proof}

\begin{rema}\label{rsmash}
1. The current version of Proposition \ref{psmash}(\ref{icopr6p}) is sufficient for the purposes of this paper; yet the following modification of this statement is quite useful also (and will be applied elsewhere).

So, assume that $\al$ is a  {\it regular} cardinal (i.e., it cannot be presented as a sum of less then $\al$ cardinals that are less than $\al$), the category $\cu$ is closed with respect to coproducts of cardinality less then $\al$, $w$ is a weight structure on $\cu$ such that $\cu_{w\ge 0}$ is closed with respect to $\cu$-coproducts of cardinality less then $\al$, $H'$ is a cohomological functor from $\cu$ into $ \au$ that converts coproducts of this sort into products, and products of cardinality less then $\al$ are exact on $\au$. Then the obvious modification of the proof of  Proposition \ref{psmash}(\ref{icopr6p}) yields that all virtual $t$-truncations of $H'$ also convert  $\cu$-coproducts of cardinality less then $\al$ into products.

2. The author suspects that is suffices to assume that $\au$ is an AB4 category in part \ref{icopr6n} of our proposition. At least, this is easily seen to be the case for $\al=\alz$, i.e., for functors that respect coproducts (cf. parts \ref{ismash2} and \ref{icopr6p} of the proposition). However, below we will only need the  $\au=\ab$ case of the statement.
\end{rema}

\subsection{On the existence of $t$-structures orthogonal to smashing weight structures}\label{smashort} 

To formulate the main results of this section we will need some more definitions.

\begin{defi}\label{dcomp}
Let  $\cu$ be a smashing triangulated category, $\cp$ is a subclass of $ \obj \cu$, $\cu'$ is an arbitrary triangulated category, and $\cp'\subset \obj \cu'$. 

\begin{enumerate}
\item\label{idbrown} We will say that  $\cu$  satisfies the  {\it Brown representability} property whenever any cp functor from $\cu$ into $\ab$ is representable. 

Dually, we will say that 
 $\cu'$ satisfies the  {\it dual Brown representability} property if $\cu'$ is smashing  and any  functor  from $\cu'$ into $\ab$ that respects products is corepresentable (i.e., if $\cu'{}\opp$ satisfies the   Brown representability assumption).

\item\label{dsmall}
For an infinite cardinal $\al$ an object $M$ of $\cu$ is said to be {\it $\al$-small} (in $\cu$) if the functor $H^M=\cu(M,-):\cu\to \ab$ 
 is $\al$-small (see Definition \ref{dsmash}(\ref{idal})).

Moreover, $\alz$-small objects of $\cu$ (corresponding to  functors that respect coproducts) will also said to be {\it compact}.

\item\label{dlocal}
We will say  that a full strict triangulated subcategory $\du\subset \cu$ is  {\it localizing}  whenever it is closed with respect to $\cu$-coproducts. 
Respectively, we will call the smallest localizing  subcategory of $\cu$ that contains a given class $\cp\subset \obj \cu$  the {\it  localizing subcategory of $\cu$ generated by $\cp$}.

\item\label{dcgen}
We will say that $\cp$ {\it compactly generates} $\cu$ and that $\cu$ is compactly generated if $\cp$ generates $\cu$ as its own localizing subcategory  and $\cp$ is a {\bf set} of compact objects of $\cu$.

Moreover, we will say that a subcategory $C$ of $\cu$ compactly generates $\cu$ whenever $C$ is essentially small and (any) its small skeleton compactly generates $\cu$.

\item\label{dgenw} We will say that  $\cp'$ {\it generates} a weight structure $w$ on  $\cu'$ whenever $\cu'_{w\ge 0}=(\cup_{i>0}\cp'[-i])\perpp$.

\item\label{dgent} We will say that 
 $\cp'$ {\it generates} a $t$-structure $t$ on $\cu'$ whenever $\cu'_{t\le 0}=(\cup_{i>0}\cp'[i])\perpp$.
\end{enumerate}
\end{defi}

\begin{rema}\label{rcomp}
1. Recall that $\cu$ satisfies both the Brown representability property and its dual whenever it is  compactly generated; 
 see Proposition 8.4.1, Proposition 8.4.2, Theorem 8.6.1, and  Remark 6.4.5  of \cite{neebook}.  Moreover, the Brown representability property is fulfilled whenever $\cu$ is just {\it $\alo$-perfectly generated} (see Definition 8.1.4 and Theorem 8.3.3 of ibid.). 

Recall also that any  triangulated category possessing a combinatorial (Quillen) model satisfies the dual Brown representability property; see \S0 of \cite{neefrosi} (the statement is given by the combination of Theorems 0.17 and 0.14 of ibid.).

Furthermore, the abundance of smashing triangulated categories that satisfy the Brown representability property has motivated the author not to consider orthogonal $t$-structures on $\cu'\neq \cu$ in Theorem \ref{tsmash}(I) below (in contrast to Theorem \ref{tsatur}(II); note however that one can easily formulate and prove the corresponding modification of Theorem \ref{tsmash}(I)).

2. The easy Lemma 4.1.4 of \cite{neebook} says that for any infinite cardinal $\al$ the class of $\al$-small objects gives a (full strict) triangulated subcategory $\cu^{(\al)}$ of $\cu$. Moreover, if $\al$ is  regular 
 (see Remark \ref{rsmash}(1)) then $\cu^{(\al)}$ is closed with respect to $\cu$-coproducts of less than $\al$ objects; see Lemma 4.1.5 of ibid. 

On the other hand, an object $M$ of $\cu$ is $\al$-small if and only if any morphism from $M$ into $\coprod_{i\in I} N_i$ factors through $\coprod_{j\in J} N_j$ for some $J\subset I$ of cardinality less than $\al$. Now, take $M=\coprod_{i\in I} M_i$
where all $M_i$ are non-zero and $I$ is of cardinality $\al$. Then $\id_M$ does not possess a  factorization through any $\coprod_{j\in J} M_j$ for $\# J<\al$; hence $M$ is not $\al$-small. These observations demonstrate that the filtration of $\cu$ by   $\cu^{(\al)}$ is "often non-trivial". Note moreover that any object of $\cu$ is $\al$-small for some cardinal $\al$ if $\cu$ is {\it well-generated}, whereas this is the case whenever $\cu$ possesses a combinatorial  model (by Proposition 6.10 of \cite{rosibr}; cf. part 1 of this remark).

3.  A class $\cp'$ as above is easily seen to determine weight and $t$-structures it generates on $\cu'$ (if any) completely; see either of Proposition 2.4(1) (along with \S3) of  \cite{bvt} or  Remark \ref{rtst}(3)  and Proposition \ref{pbw}(\ref{iort}) above.
\end{rema}

Now we prove our first "practical" existence of $t$-structures results; see Definitions \ref{dsmash} and \ref{dcomp} for the 
 notions mentioned in our theorem.

\begin{theo}\label{tsmash}
Let $w$ be a weight structure on $\cu$.

I. Assume 
 that $\cu$ satisfies the Brown representability property (and so, $\cu$ is smashing). 

Then there exists a $t$-structure $t^r$ right adjacent to $w$ if  and only if $w$ is smashing. Moreover, $t^r$ is cosmashing (if exists; see Definition \ref{dsmash}(\ref{ismt})) and its heart is equivalent to the category of those  additive functors $\hw\opp\to \ab$ that respect products.

II. Assume  that $w$ is cosmashing and  $\cu$ satisfies the dual Brown representability property.  

1. Then there exists a smashing $t$-structure $t^l$ left adjacent to $w$ and $\hrt^l$ is equivalent to the category of those additive functors $\hw\to \ab$ that respect products.

2. Assume that $w$ is also smashing. Then for any infinite cardinal $\al$ the weight structure $t^l$ given by the previous assertion restricts to the subcategory $\cu^{(\al)}$ of $\cu$ (see Remark \ref{rcomp}(2) and Definition \ref{dtstro}(5));  
 this restricted $t$-structure 
 $t^{(\al)}$  is the only $t$-structure on $\cu^{(\al)}$ that is left orthogonal to $w$.  
\end{theo}
\begin{proof}
I. The "only if" assertion is essentially given by Proposition 2.4(6) of \cite{bvt} (cf. \S3. of ibid.; 
the statement is also very easy for itself).

Conversely,  assume that $w$ is smashing. According to Proposition \ref{prefl}(\ref{irefl1})  we can  apply Theorem \ref{trefl} (in the case $\du=\cu'=\cu$) to obtain 
that the existence of $t^r$ is equivalent to the representability of $\tau_{\le 0}H_{M'}$ for any representable functor $H_{M'}$. Next, Proposition \ref{psmash}(\ref{icopr6p}) says that   $\tau_{\le 0}H_{M'}$ is a cp functor since $H_{M'}$ is. Hence $\tau_{\le 0}H_{M'}$ is representable by the Brown representability assumption, and we obtain  that $t^r$ exists indeed. 

Next,  the category $\cu$  is cosmashing according to  Proposition 8.4.6 of  \cite{neebook} (since it satisfies the Brown representability property).  Moreover, $t^r=(C_1,C_2)$, where $C_1= \cu_{w\ge 1}^{\perp}$  and $C_2= \cu_{w\le -1}^{\perp}$ according to Proposition \ref{prefl}. 
 Hence the class $\cu_{t^r}\ge 0$  is closed with respect to $\cu$-coproducts; thus $t^r$ is cosmashing as well. 

Lastly, since $t^r=(C_1,C_2)$, the class $\cu_{t=0}$ equals $(\cu_{w\ge 1}\cup \cu_{w\le -1})^{\perp}$; hence $\hrt$ can be calculated using Proposition 2.3.2(8) of \cite{bwcp} (see also Remark 2.1.3(2) of ibid. and Proposition \ref{pwrange}(\ref{iwrpure}) above).  

II.1. This is just the categorical dual to assertion I.

2. The uniqueness of a $t$-structure on $\cu^{(\al)}$ that is left orthogonal to $w$ is given by (the dual to) Proposition \ref{prefl}. 
 Next, for any object $M$ of $\cu^{(\al)}$ 
the functor $H^M$ is $\al$-small by the definition of  $\cu^{(\al)}$. Now, for $M'=t^l_{\ge 0} M$ Proposition \ref{pwfil}(5) says that  $H^{M'}\cong \tau_{\ge 0}H^M$. Hence the   functor $H^{M'}$ is $\al$-small as well according to Proposition \ref{psmash}(\ref{icopr6n}), and we obtain that $M'$ is an object of $\cu^{(\al)}$. Thus $M$ has a $t^l$-decomposition whose components are objects  of $\cu^{(\al)}$; therefore $t^l$ restricts to $\cu^{(\al)}$ indeed.
\end{proof}

\begin{rema}\label{rsmashex}
1.  Now let us discuss examples to Theorem \ref{tsmash}(I). 

According to Theorem 5 of \cite{paucomp}, 
 any set $\cp$ of compact objects of $\cu$ generates a (unique) smashing weight structure (see Definition \ref{dcomp}(\ref{dgenw}) and Remark \ref{rcomp}(3)). Moreover,  Theorem 4.5(2) of \cite{postov} and Theorem 4.2.1(1,2) of \cite{bpgws} 
(we will mention these statements in the the proof of Corollary \ref{ckeller}(2) below) enable one to check  whether two weight structures obtained this way are distinct. Thus one may say that there are lots of smashing weight structures on $\cu$ whenever there are "plenty" of compact objects in it (see Theorem 4.15 of \cite{postov} for a certain justification of this claim for derived categories of commutative rings).  Thus part I of our theorem yields a rich collection of $t$-structures, and the author does not know of any other methods that give all of them (cf. Remark \ref{rneetgen}).

2.  Recall (from Proposition 3.4(4) of \cite{bvt})  that "shift-stable" weight structures are in one-to-one correspondence with exact embeddings $i:L\to \cu$ possessing  right adjoints.  Hence applying our theorem  in this case we obtain the following: if 
 $i$ possesses a right adjoint respecting coproducts and $\cu$ satisfies the Brown representability property then for the full triangulated subcategory $R$ of $\cu$ with $\obj R=L\perpp$ the embedding $R\to \cu$ possesses a right adjoint as well. 
Thus $R$ is {\it admissible} in $\cu$ in the sense of  \cite{bondkaprserr} and the embedding $R\to\cu$ may be completed  to a {\it gluing datum} (cf. \cite[\S1.4]{bbd} or \cite[\S9.2]{neebook}).

So we re-prove Corollary 2.4 of \cite{nisao}. 

3. The author suspects that the heart of the restricted $t$-structure $t^{(\al)}$ 
 in part II.2 of our theorem can be computed similarly to the hearts in assertions I and II.1, and so using the theory of {$w$-pure functors} as developed in \S2 of \cite{bwcp}. 
\end{rema}

Let us now verify that 
Theorem 3.1 of \cite{kellerw}  (that essentially generalizes Theorem 3.2 of  \cite{konk}) gives an example for the setting of Theorem \ref{tsmash}(II.2), and study the corresponding structures in detail.

\begin{coro}\label{ckeller}
Let $\au$ be an essentially small abelian semi-simple subcategory of the subcategory $\cu^{(\alz)}$ of $\cu$  that generates $\cu$ as its own localizing subcategory,  and assume that $\obj \au\perp_{\cu}\cup_{i<0}\obj\au[i]$. 

1. Then there exist a smashing and cosmashing weight structure $w$ and a $t$-structure $t$ on $\cu$  that are generated by $\obj \au$, and $t$ is left adjacent to $w$. 

2. $t$ restricts to the subcategory $\cu^{(\al)}$ (see 
 Remark \ref{rcomp}(2)) for any infinite cardinal $\al$. Moreover,  in the corresponding couple $t^{(\alz)}=(\cu^{(\alz)}_{t^{(\alz)}\le 0}, \cu^{(\alz)}_{t^{(\alz)}\ge 0})$ the class $\cu^{(\alz)}_{t^{(\alz)}\le 0}$ (resp. $\cu^{(\alz)}_{t^{(\alz)}\ge 0}$) is the envelope of $\cup_{i\le 0}\obj \au[i]$ (resp. of $\cup_{i\ge 0}\obj \au[i]$) in $\cu$ (see \S\ref{snotata}).
\end{coro}
\begin{proof}
1.  Since $\au$ is semi-simple, the category $\adfu(\au,\ab)$ is semi-simple as well. Thus we can apply Theorem 3.1 of \cite{kellerw} to obtain the existence of a weight structure $w$ on $\cu$  such that $\cu_{w\ge 0}=(\cup_{i>0}\cp[-i])\perpp$ (i.e. $w$ is generated by $\cp=\obj \au$) and  $\cu_{w\le 0}=(\cup_{i>0}\cp[i])\perpp$.
 Since the category $\cu$ is compactly generated by $\au$, it satisfies the dual Brown representability property (by the aforementioned Theorem 8.6.1 and  Remark 6.4.5  of \cite{neebook}).  Next, $w$ is obviously smashing and cosmashing. Applying Theorem \ref{tsmash}(II.1), we obtain the existence of a smashing $t$-structure $t$ that is left adjacent to $w$. Since $\cu_{w\le 0}=\cu_{t\le 0}$, we obtain that $t$ is generated by $\cp$ (as a $t$-structure) as well.    

2. $t$ restricts to the subcategory $\cu^{(\al)}$ for any infinite cardinal $\al$ according to part II.2 of Theorem \ref{tsmash}.
Thus it remains to prove that the classes $\cu_{t\le 0}\cap\obj \cu^{(\alz)}=\cu_{w\le 0}\cap\obj \cu^{(\alz)}$ and 
  $\cu_{t\ge 0}\cap\obj \cu^{(\alz)}$ are the envelopes in question. The latter statement an easy consequence of Theorem 4.2.1(2) of \cite{bpgws} applied to $w$ and $t$ respectively (see Remark 4.2.2(1,2) of ibid.; note also that Theorem 4.5(1,2) of \cite{postov} gives this statement in the case under the assumption that $\cu$ is a "stable derivator" triangulated category).  \end{proof}

\begin{rema}\label{rkeller}
\begin{enumerate}
\item\label{ikel1}
Thus we obtain a serious generalization of 
 the existence of a (certain) $t$-structure on $\cu^{(\alz)}$ 
 part of   \cite[Theorem 7.1]{kellerw}.

Note also that the assumption that $\cu$ is compactly generated by $\au$ does not appear to be necessary in Theorem 3.1 of ibid.; thus is may be omitted in our corollary as well. However, in this case it is probably more interesting to look at the  localizing
subcategory $\cu'$ of $\cu$ generated by $\obj \au$ instead. We  note that  the embedding $i:\cu'\to \cu$ possesses an exact right adjoint $F$ (since $\cu'$ satisfies the Brown representability property; see Theorem 8.4.4 and  Lemma 5.3.6  of ibid.), and $F$ allows to recover $w$ and $t$ from their restrictions to $\cu'$ (whose existence is given by the present form of Corollary \ref{ckeller}) via Propositions 3.2(5) and 3.4(3) of \cite{bvt}. 

\item\label{ikel2} Now we try to study the question which 
 $t$-structures on $\cu^{(\alz)}$ extend to examples for our corollary.

So, assume that $\cu$ is an arbitrary triangulated category 
 and $t'$ is a $t$-structure on $\cu^{(\alz)}$, and take $\au'=\hrt'$. 
Then  we have $\obj \au' \perp(\cup_{i>0}\obj \au'[-i])$ by the orthogonality axiom for $t'$.

Thus any 
essentially small abelian subcategory $\au$  of $\hrt'$ whose objects are semi-simple satisfies all the assumptions of our corollary except the one that $\au$  compactly generates $\cu$. Hence we can apply our corollary to the  localizing subcategory $\cu'$ of $\cu$ generated by $\obj \au$. 

\item\label{ikel3} Now assume in addition that $\cu$ is compactly generated, $t'$ is bounded, and $\hrt'$ is a length category (cf. Theorem 7.1 of \cite{kellerw}). Then the category $\cu^{(\alz)}$ is essentially small according to Lemma 4.4.5 of \cite{neebook}; 
hence $\hrt'$ is essentially small as well, and we can take $\au$ to be its subcategory of semi-simple objects. 

Since $t'$ is bounded and $\hrt'$ is a length category,  
 the category $\cu^{(\alz)}$ is densely generated by $\obj \au$; hence  $\au$ is easily seen to generate  $\cu$ as its own localizing subcategory. Thus one can apply Corollary \ref{ckeller} to this setting.
Moreover, it is easily seen that our assumptions on $t'$ (combined with part 2 of our corollary) imply that the corresponding $t$-structure $t^{\alz}$ coincides with $t'$.

\item\label{ikel4} It would be interesting to find which assumptions on a general $t$-structure $t'$ on $\cu^{(\alz)}$ ensure that $t'$ "extends" to a $t$-structure on $\cu$ and a weight structure $w$ that is right adjacent to $t$.

Suppose that $\cu$ is compactly generated (or at least that $\cu^{(\alz)}$ is essentially small). Then it appears to be quite reasonable to consider the weight structure $w$ that is generated (in $\cu$) by $\cu^{(\alz)}_{t'\ge 0}$ (see Definition \ref{dcomp}(\ref{dgenw}) and Remark \ref{rsmashex}(1)) and the $t$-structure $t$ that is generated by $\cu^{(\alz)}_{t'\le 0}$ (see Definition \ref{dcomp}(\ref{dgent}); the existence of $t'$ is provided by Theorem A.1 of \cite{talosa}). However, it is not clear how to check whether $t$ is left adjacent to $w$ in the general case.

\item\label{ikel5} Let us now describe an example that demonstrates that certain additional assumptions are necessary to ensure that this $t$ is left adjacent to $w$. Possibly, for this purpose it suffices to impose certain conditions on $\hrt$ (for $t$ as above), but it would certainly be more interesting to seek for formulations in terms of $\hrt'$. 

So, we take $R$ to be a left semi-hereditary ring (see Definition \ref{dfp}(\ref{idfp3}) below or \S0.3 of \cite{cohn}) that is not noetherian; in particular, one can take $R$ to be any non-noetherian valuation ring (see Proposition \ref{pfp}(\ref{ipfp1})). We will consider left $R$-modules only, and set $\cu=D(R)$, i.e., $\cu$ is the derived category of (left) $R$-modules. Then $\cu^{(\alz)}$ 
is the subcategory of {\it perfect complexes} in $D(R)$ (a well-known fact; see Proposition \ref{pfp}(\ref{ipfp6}) below), i.e., its objects are quasi-isomorphic to bounded complexes of finitely generated projective $R$-modules. 

Now we claim that the canonical $t$-structure $t$ on $\cu$ (i.e., $t$-truncations are the canonical truncations of complexes, and $t$-homology is the "$R$-module" one) restricts to $\cu^{(\alz)}$. 
Indeed,  objects of $\cu^{(\alz)}$ are $t$-bounded; hence it suffices to verify that the cohomologies of perfect complexes are  
 perfect themselves. Now, these cohomology modules are finitely presented (see Definition \ref{dfp}(\ref{idfp1}) and Proposition \ref{pfp}(\ref{ipfp2},\ref{ipfp3}) below; cf. also Theorem A.9 of \cite{cohn}), whereas for any finitely presented $R$-module $M$ one can take a surjection $R^n\to M$, and its kernel is projective and finitely presented according to Proposition \ref{pfp}(\ref{ipfp5}).

We will write $t'$ for the restriction of $t$ to $\cu^{(\alz)}$.
It remains to verify for $\cp=\obj \hrt'$ that  $\cu_{w\ge 0}=(\cup_{i>0}\cp[-i])\perpp$) and  $\cu_{w\le 0}=(\cup_{i>0}\cp[i])\perpp$ is not a weight structure. Assume the opposite; then $\hw$ obviously contains injective $R$-modules (placed in degree $0$) and is closed with respect to $\cu$-coproducts. Since $R$ is not noetherian, the Bass-Papp Theorem 1.1 of \cite{bass} implies that $\hw$
also contains a non-injective $R$-module $M$. Then we take an embedding $M\to I$ of $M$ into an injective module, and consider the corresponding distinguished triangle $M\to I \to N\to M[1]$. Since $N$ is an $R$-module (put in degree $0$) as well, we obtain $N\in \cu_{w\ge 0}$. Hence $N\perp M[1]$ and we obtain that $M$ is a retract of $I$. Thus  is an injective $R$-module, and we obtain a contradiction.

More generally, it 
is sufficient to assume that $R$ is any {\it coherent} ring such that any finitely presented module over it possesses a bounded projective resolution; see Proposition \ref{pfp}(\ref{ipfp4}) below. The author suspects that rings of the form $R'[x_1,x_2,\dots,x_2]$, where $R'$ is commutative semi-hereditary, give examples of these assumptions.

\end{enumerate}
\end{rema}


\subsection{Weight structures extended from subcategories of compact objects, and orthogonal $t$-structures}\label{sort}

Now we prove that weight structures extend from subcategories of compact objects to the localizing subcategories they generate. 
The (proof of the) first part of the following theorem is quite similar to the corresponding arguments in \S2 of \cite{bsnew}.
Note also that one can certainly assume that $\eu$ is compactly generated and thus equals $\du$.

\begin{theo}\label{tcompws}
Let $\eu$ be a smashing triangulated category, $\cu=\eu^{(\alz)}$, $\du$ is the localizing subcategory of $\eu$ generated by $\obj \cu$, and $w$ is a weight structure on $\cu$.

I.1. Then $w$ extends uniquely to a smashing weight structure $w_{\du}$ on $\du$, i.e., $\du_{w_{\du}\le 0}\cap \obj \cu=\cu_{w\le 0}$ and $\du_{w_{\du}\ge 0}\cap \obj \cu=\cu_{w\ge 0}$.

2. $\du_{w_{\du}\le 0}$ (resp.  $\du_{w_{\du}\ge 0}$) is the smallest extension-closed subclass of $\obj \du$ that is closed with respect to $\du$-coproducts and contains $\cu_{w\le 0}$ (resp.  $\cu_{w\ge 0}$).

3.  $\du_{w= 0}$ consists of all retract of all (small) $\du$-coproducts of elements of $\cu_{w= 0}$.

II. Assume in addition that the category $\cu$ is essentially small. 

1. Then there exists a $t$-structure $t_{\du}$ right adjacent to $w_{\du}$ (on $\du$); it is right orthogonal to $w$.

2. 
 $t_{\du}$ is strictly right orthogonal to $w$; hence $t_{\du}$ is both smashing and cosmashing.

3. 
$\hrt_{\du}$ is equivalent (in the obvious way) to the category $\adfu(\hw\opp,\ab)$.

4. For an infinite cardinal $\be$ consider the full subcategory $\du_\be$ of $\du$ 
that consists of those $N$ such that $\#\du(M,N)<\be$ for all $M\in \obj \cu$.

Then this subcategory is triangulated, $t_{\du}$ restricts to it, and the heart of this restriction $t_\be$ is naturally equivalent to the category of those functors from $\hw\opp$ into $\ab$ whose values are of cardinality less than $\be$.
\end{theo}
\begin{proof}
I.1,2. We will write  $C_1$ and $C_2$ for the classes of objects described in assertion I.2. Let us prove that 
 $(C_1,C_2)$ is a weight structure on $\du$ indeed. 

Firstly, 
 axiom (ii) of Definition \ref{dwstr} (for $w$) easily implies that  $C_1\subset C_1[1]$ and $C_2[1]\subset C_2$.  
Combining this statement with  
Proposition \ref{pstar}(II.2) below we obtain that $C_1$ and $C_2$ are retraction-closed in $\du$. 

Next, the compactness of the elements of $\cu_{w\le 0}$ in $\du$ implies that the class $\cu_{w\le 0}\perpp$ is closed with respect to coproducts. Since it is also extension-closed and contains $\cu_{w\ge 1}$ by the axiom (iii) of Definition \ref{dwstr}, this orthogonal 
 contains $C_2[1]$, i.e., $\cu_{w\le 0}\perp C_2[1]$. Thus $\cu_{w\le 0}\subset \perpp C_2[1]$, and since the latter class is closed with respect to coproducts and extension, we obtain that $C_1\perp C_2[1]$ (cf. the proof of \cite[Lemma 1.1.1(2)]{bokum}).

Let us now prove the existence of weight decompositions, i.e., for the set $E$ of those $M\in \obj \du$ such that there exists a distinguished triangle $LM\to M\to RM\to LM[1]$ with $LM\in C_1$ and $RM\in C_2[1]$ we should prove $E=\obj \du$. Now, $E$ certainly contains $\obj \cu$, and 
 Proposition \ref{pstar}(I, II.1) below implies that it is also extension-closed and closed with respect to coproducts. Hence $E=\obj \du$, and we obtain that $w_{\du}=(C_1,C_2)$ is a weight structure on $\du$ indeed. Certainly, this weight structure is smashing.

Lastly, assume that $v$ is a smashing weight structure on $\du$ such that $ \cu_{w\le 0}\subset \du_{v\le 0}$ and $ \cu_{w\ge 0}\subset \du_{v\ge  0}$. Then we obviously have $ \du_{w_{\du}\le 0}\subset \du_{v\le 0}$ and $ \du_{w_{\du}\ge 0}\subset \du_{v\ge  0}$, and applying Proposition \ref{pbw}(\ref{iuni}) we obtain $v=w_{\du}$.

3. Denote  our candidate for $\du_{w_{\du}=0}$ by $C$. 
Firstly we note that  $C\subset \du_{w_{\du}=0}$ since the latter class is  closed with respect to coproducts according to Proposition \ref{psmash}(\ref{icoprh}). 

Applying Proposition \ref{pbw}(\ref{isplit})  we obtain that $C$  is  extension-closed (since this class is certainly additive); certainly, it is also closed with respect to coproducts. 

Next we apply 
  Proposition \ref{pstar}(I, II.1) once again to obtain that the class of extensions of elements of $\du_{w_{\du} \ge 1}$ by that of $C$ is extension-closed and closed with respect to coproducts; hence this class coincides with $\du_{w_{\du} \ge 0}$. Thus for any $M\in \du_{w_{\du}= 0}$ there exists its weight decomposition $LM\to M\to RM\to LM[1]$ with $LM\in C$.  Since $M$ is a retract of $LM$ according to Proposition \ref{pbw}(\ref{iwdmod}),  we obtain that $M\in C$.

II.1. The category $\du$ is compactly generated by $\cu$ in this case; hence $\du$ satisfies the Brown representability property (see Remark \ref{rcomp}(1)). Next, $w_{\du}$ is smashing; thus a $t$-structure $t_{\du}$ adjacent to it exists according to Theorem \ref{tsmash}(I). Certainly, $t_{\du}$ is also 
 orthogonal to $w$.

2. $t_{\du}$ is cosmashing according to Theorem \ref{tsmash}(I).

Next, Proposition \ref{prefl} implies that $t_{\du}$ is strictly right orthogonal to $w_{\du}$, i.e., $\du_{t_{\du}\le 0}=\du_{w_{\du}\ge 1}\perpp$ and $\du_{t_{\du}\ge 0}=\du_{w_{\du}\le -1}\perpp$. Since for any object $N$ of $\du$ the class ${}^{\perp_{\du}} N$ is closed with respect to coproducts and extensions, assertion I.2 implies that $\du_{w_{\du}\ge 1}\perpp=\cu_{w\ge 1}\perpp$ and $\du_{w_{\du}\le -1}\perpp =\cu_{w\le- 1}\perpp$. Lastly, the elements of $\cu_{w\ge 1}$ are compact in $\du$; hence $t_{\du}$ is smashing as well.

3. According to Theorem \ref{tsmash}(I), the category $\hrt_{\du}$ is equivalent to the category of those functors from $\hw_{\du}\opp$ into $\ab$ that respect products. Thus it remains to apply the description of $\hw_{\du}$ provided by assertion I.3.

4. The proof relies on the following obvious observation that will be denoted by (*): the category of abelian groups of cardinality less than $\be$ is a (full) exact abelian subcategory of $\ab$. It immediately implies that the category $\du_{\be}$ is triangulated.  
	
	To prove that $t$ restricts to $\du_{\be}$ we argue similarly to the proof of Theorem \ref{tsmash}(II.2). For an object 
	$N$ of $\du_{\be}$  the values of the  corresponding functor $H_N:\cu\opp\to \ab$ are of cardinality less than $\be$  by the definition of $\du_{\be}$, and 
	 combining (*) with Proposition \ref{pwfil}(7) we obtain that the functor   $\tau_{\ge 0}H_N$ possesses this property as well. Since 
 	for $N'=t_{{\du},\ge 0} N$ we have $\tau_{\ge 0}H_N\cong H_{N'}$, $N'$ is an object of $\du_{\be}$.  Thus $N$ has a $t_{\du}$ 
 -decomposition whose components are objects  of $\du_{\be}$, i.e.,  $t_{\du}$ restricts to $\du_{\be}$.

It remains to calculate the heart of the $t$-structure $t_{\be}$ obtained. Applying assertion II.3 we obtain that it suffices to verify the following: a $w$-pure functor $\cu\opp\to \ab$ has values of cardinality less than $\be$ if and only if the values of its restriction to $\hw$ satisfy this property. This is immediate from Lemma 2.1.4 of \cite{bwcp} along with (*). 
\end{proof}

\begin{rema}\label{restralcomp}
1. The restriction of our theorem to the case where $w$ is bounded and $\cu$ is essentially small was essentially established in \S4.5 of \cite{bws}; cf.  Theorem 3.2.2 of \cite{bwcp} and Remark 2.3.2(2) of \cite{bsnew} for some more detail.

2. Similarly to Corollary 2.3.1(2) of ibid., for any regular cardinal $\al$ the weight structure $w$ restricts to the smallest subcategory of $\du$ that contains $\cu$ and is closed with respect to coproducts of cardinality less than $\al$. Moreover, this filtration of $\du$ may be easily completed to a filtration indexed by all infinite cardinals, cf. loc. cit.

Moreover, part I of our theorem along with Remark \ref{rkeller}(\ref{ikel5}) demonstrate that it is "easier to extend weight structures from compact objects than $t$-ones". 

3. Note that one can easily obtain plenty of examples  for our theorem such that  $w$ is unbounded. 

Indeed, can obtain lots of unbounded weight structures on essentially small triangulated categories using (say) the previous part of this remark.  Next, it appears that any "reasonable" essentially small triangulated category $\cu$ is the subcategory of compact objects in some smashing triangulated category $\du$; cf. Remark 5.4.3(1) and Proposition 5.5.2 of \cite{bpgws} (that relies on \cite{tmodel}; one should dualize it and pass to a subcategory).

4.   Theorem 
B of \cite{krauwg} relates  the filtration of $\du$ by the subcategories $\du_{\be}$ to the so-called $\be$-compactness filtration (as introduced in \cite{neebook}).
\end{rema}

\section{On 
 $t$-structures related to saturated categories and coherent sheaves}\label{scoh}

 In \S\S\ref{satur}--\ref{sperf2} we will treat $R$-linear categories and functors. 
we will always assume that $R$ is an associative commutative unital {\it coherent} (see Definition \ref{dfp} below) ring; moreover, $R$ will be Noetherian in \S\S\ref{sperf1}--\ref{sperf2}.

We start \S\ref{satur} from recalling (from \cite{neesat} and preceding papers on the subject) the definition of $R$-saturated   categories and (locally) finite $R$-linear functors.
Next  we prove the existence of a $t$-structure adjacent to a bounded weight structure on a $R$-saturated category $\cu$; we also generalize this statement to the case of orthogonal structures.  

In \S\ref{sperf1} we describe  rich families of "geometric" examples to these statements; one takes $\cu$ and $\cu'$ to be the derived categories of perfect complexes and of bounded (above) complexes of coherent sheaves over $X$, where $X$ is proper over $\spe R$; 
 we also apply duality if $X$ is regular (or Gorenstein).

In \S\ref{sperf2} we discuss whether one can obtain interesting results 
 "starting from" $\cu=D^b_{coh}(X)$ (instead of $\cu=D^{perf}(X)$).

In \S\ref{scoher} we recall certain definitions and statements related to coherent rings and perfect complexes.

\subsection{Constructing $t$-structures corresponding to (locally) finite functors into $R$-modules}\label{satur}

To help the reader we recall that all Noetherian rings are coherent, and if $R$ is Noetherian then an $R$-module is finitely presented if and only if it is finitely generated. Moreover, the restriction of Proposition \ref{pfp} below to the case of Noetherian rings is very well-known. We mention coherent rings in this section just for the sake of generality. Since in the most interesting of the currently available examples for our statements the ring $R$ is Noetherian, the reader may restrict herself to the Noetherian ring case.

\begin{defi}\label{dsatur} 
Let  $\cu$ be  $R$-linear category.

1. We will say that an $R$-linear cohomological functor $H$ from $\cu$ into $R-\modd$ is  {\it (anti) locally finite } 
 whenever for any $M\in \obj\cu$ the $R$-module $H(M)$ is finitely  presented and $(H(M[-i])=) H^i(M)=\ns$  for  $i\ll 0$ (resp.  for  $i\gg 0$) 

Moreover, we will say that $H$ is finite if it is both locally and anti-locally finite.


2. We will say that $\cu$ is {\it $R$-saturated} if the  representable functors from $\cu$ into $R-\modd$  are exactly all the  finite ones.

3. The symbol $\adfur(C,D)$ will denote the (possibly, big) category of $R$-linear (additive) functors from $C$ into $D$ whenever $C$ and $D$ are $R$-linear categories.   

4. We will write $R-\mmodd$ for the category of finitely presented $R$-modules  (see Definition \ref{dfp}\ref{idfp1}) below).
\end{defi}

The following statement should be understood as the conjunction of three its versions. To obtain the "finite" version
 one should ignore all adjectives in brackets, to obtain  the "locally finite" version one should take the first adjectives in all the brackets  into account, and one should take the second adjectives to get the "anti-locally finite" version. This rule should also be applied to some of the sentences in the proof. 

\begin{theo}\label{tsatur}
Assume that $\cu$ and $\cu'$ are full $R$-linear triangulated subcategories of an $R$-linear triangulated category $\du$ (recall that $R$ is a coherent ring),  $\cu$ is endowed with a bounded (resp. bounded below, resp. bounded above) weight structure $w$, and a subcategory $\cu'$ of $\cu$ is characterized by the following condition:  for $N\in \obj\du $ the $\du$-Yoneda-functor $H_N:\cu\opp\to R-\modd$ (see Definition \ref{dvtt}(3))  
 is  finite   (resp. locally finite, resp. anti-locally finite) if and only if  $N$ is an object of $\cu'$. 

I. Then all virtual $t$-truncations of   finite  (resp. locally finite, resp. anti-locally finite) functors $\cu\opp\to R-\modd$ are  finite  (resp. locally finite, resp. anti-locally finite) as well.

II. Assume that 
 finite  (resp. locally finite, resp. anti-locally finite) functors $\cu\opp\to R-\modd$ are precisely the functors of the form $H_N$  for $N\in \obj\cu'$,  and that $\cu$  reflects  $\cu'$ 
 in $\du$ (see Definition \ref{drefl}). 

1. Then $\cu'$ is a triangulated subcategory of $\du$ and there exists a (unique) $t$-structure $t$ on $\cu'$ that is strictly right orthogonal to $w$.

2. 
 The obvious Yoneda-type functor from the category $\hrt$ 
into the subcategory $\adfur(\hw\opp,R-\mmodd)$ of $\adfur(\hw\opp,R-\modd)$ is an equivalence of categories.


III. Assume that $\du$ is smashing, $\cu$ equals $\du^{(\alz)}$ and compactly generates $\du$ (thus $\cu$ is essentially small). 

1. Then the $t$-structure $t_{\du}$ on $\du$ provided by Theorem \ref{tcompws}(II) restricts to $\cu'$ (i.e., 
$t=
 ((\cu_{w\ge 1})^{\perp_{\du}}\cap \obj \cu',  (\cu_{w\le -1})^{\perp_{\du}}\cap \obj \cu')$ is a $t$-structure  on $\cu'$).

2. The obvious Yoneda-type functor from $\hrt$ into 
the category $\adfur(\hw{}\opp,R-\mmodd)\subset \adfur(\hw\opp,R-\modd)$ is an equivalence of categories.

\end{theo}
\begin{proof}
I.  Recall that virtual $t$-truncations of cohomological functors are cohomological. Moreover, virtual $t$-truncations of $R$-linear functors are obviously $R$-linear.

Now, let $H$ be a functor 
whose values are finitely presented $R$-modules.
Since the category $R-\mmodd$ is an abelian subcategory of $R-\modd$ according to Proposition \ref{pfp}(\ref{ipfp3}) below,
 the values of virtual $t$-truncations of $H$ are finitely presented as well according to Proposition \ref{pwfil}(7).\footnote{Actually, it is not necessary to assume that $R$ is coherent to prove part I of our theorem. Indeed, for any $M\in \obj \cu$ and $n\in \z$ we can complete the morphisms $w_{\le n}M\to w_{\le n+1}M$ and $w_{\ge n-1}M\to w_{\ge n}M$ to distinguished triangles to obtain that the values of 
 $\tau_{\le n }(H)$ and $\tau_{\ge n }(H)$ (see Definition \ref{dvtt}(1)) are quotients of finitely presented $R$-modules by some images of modules of this type. Thus the values of virtual $t$-truncations of $H$ are finitely presented $R$-modules (see Lemma  A.8 of \cite{cohn}).
 
On the other hand, it appears that the subcategory $\cu'$ of $\cu$ does not have to be triangulated if $R$ is not coherent.}

Next, for any $k\in \z$ and any cohomological functor $H$ from $\cu$ the functor $\tau_{\le  k}(H)$ is of weight range $\le k$ according to Proposition \ref{pwrange}(\ref{iwrvt}). Assume that $w$ is bounded below; then 
 for any fixed $k$ and $M\in \obj \cu$ we have $\tau_{\le  k}(H)(M[j])=0$ if $j$ is large enough. Hence if $H$ takes values in $R-\mmodd$  then the functor  $\tau_{\le  k}(H)$ is locally finite.  Applying the long exact sequence (\ref{evtt}) we obtain that the functor $\tau_{\ge  k+1}(H)$ is locally finite if $H$ is. 

Moreover, the corresponding boundedness statement for the case where  $w$ is bounded  above and $H$ is anti-locally finite is dual to the "locally finite" version that we have just verified. 

Lastly, combining 
the 
locally finite and the anti-locally finite cases of our assertion one immediately obtains its finite case. 

II.1. Since $R-\mmodd$ is an exact abelian subcategory of $R-\modd$  (see Proposition \ref{pfp}(\ref{ipfp1})), $\cu'$ is easily seen to be a triangulated subcategory of $\du$.

Next, combining our assumptions with assertion I we obtain that virtual $t$-truncations of functors of the type $H_N$ for $N\in \obj \cu'$ are $\du$-represented by objects of $\cu'$ as well. Thus we can combine Theorem \ref{trefl} with Proposition \ref{prefl}(\ref{irefl4}) to obtain the result. 

2. Certainly, restricting locally or anti-locally finite  functors 
from $\cu$ to $\hw$ yields additive functors from $\hw\opp$ into $R-\mmodd$. This restriction gives an embedding of $\hrt$ into 
 $\adfur(\hw\opp,R-\mmodd)$ according to Proposition \ref{prefl}(\ref{irefl4}). 

Lastly, assume that 
$A$ belongs to $\adfur(\hw\opp,R-\mmodd)$. 
 It remains to  check that the corresponding $w$-pure functor $H_A:\cu\opp\to R-\modd$ provided by Proposition \ref{pwrange}(\ref{iwrpure}) (cf. also Remark \ref{rpure}(2)) is  (anti) locally finite whenever $w$ is bounded below (resp. above).

Applying the fact that the category $R-\mmodd$ is an abelian subcategory of $R-\modd$ once again and combining it with Lemma 2.1.4 of \cite{bwcp} we obtain that the values of $H_A$ are finitely presented as well. Since  $H_A$ is of weight range $[0,0]$ and $w$ is bounded,   
$H_A$ is (anti) locally finite immediately from Proposition \ref{pwrange}(\ref{iwrpureb}).

III.1. Since $w$ is orthogonal to $t_{\du}$, Proposition \ref{pwfil}(4) says that for any $N\in \obj\du$ the functors $\du$-represented by its $t_{\du}$-truncations on $\cu$ are the corresponding virtual $t$-truncations of the functor $H_N$. Thus for any $M\in \obj \cu'$ the objects $t_{\du,\le k}M$ and $t_{\du,\ge k}M$ are objects of $\cu'$ as well according to assertion I; hence $t_{\du}$ restricts to $\cu'$ indeed. 

2. The full faithfulness of the functor $\hrt\to \adfur(\hw\opp,R-\modd)$ is immediate from Theorem \ref{tcompws}(II.3). Moreover, this functor obviously factors through $\adfur(\hw{}\opp,R-\mmodd)$, and arguing similarly to the proof of assertion II.2 we  obtain the equivalence in question. 
\end{proof}

\begin{rema}\label{rsatur}
\begin{enumerate}
\item\label{idercoh} In 
 the examples in \S\ref{sperf1} and \S\ref{sperf2} below 
 the ring $R$ will always be noetherian. So let us demonstrate that non-noetherian coherent rings are also actual (at least) in the context of parts I and III of our theorem.

Similarly to Remark \ref{rkeller}(\ref{ikel5}), we take $R$ to be an arbitrary (not necessarily noetherian) coherent  ring which we have to assume to be commutative here, set $\du=D(R)$; then $\cu=\du^{(\alz)}$ 
is the subcategory of  perfect complexes (see Proposition \ref{pfp}(\ref{ipfp6}) below). 

Hence $\cu$ is equivalent to $K^b(\operatorname{Proj_{fin}}R)$, where $\operatorname{Proj_{fin}}R$ is the category of finitely generated projective $R$-modules, and we set $w$ to be the "stupid" weight structure whose heart (essentially) equals $\operatorname{Proj_{fin}}R$; see Remark \ref{rstws}(1). Since $w$ is bounded, we can apply 
 any of the versions of (parts I and III) of our theorem 
 if we take $\cu'$ to be equal to the category $\cu'_i\subset \du$ for $1\le i\le 3$; here $\cu'_1$ corresponds to  finite cohomological functors from $\cu$ into $R-\modd$, $\cu'_2$ corresponds to locally finite functors, and $\cu_3$ corresponds to anti-locally finite ones. Moreover, it is easily  seen that $\cu'_1\cong D^b(R-\mmodd)$, $\cu'_2\cong D^-(R-\mmodd)$, and $\cu'_3\cong D^+(R-\mmodd)$.
Certainly, the corresponding $t$-structures are the canonical ones. 

This example is certainly closely related to Proposition \ref{pnee1}(1) below; we put it here just to demonstrate that it may be actual to consider the case where $R$ is coherent but not noetherian.

\item\label{ireg} The question whether $\cu=\cu'_1$ in the notation that we have just introduced is well-known to be equivalent to the regularity of $R$.  The reader is recommended to look at the paper \cite{ksosnreg}  for an interesting discussion of this matter (in the language of adjacent structures) in the more general setting of modules over ring spectra.



\item\label{iex} In  \S\ref{sperf1} 
 we will discuss interesting "geometric" examples for our theorem. Unfortunately, this does not include any examples for the locally finite and anti-locally finite versions of part II of the theorem. 


It is also worth noting that  bigger families of examples can be obtained by means of Theorem 0.3 of \cite{neesat} and Theorem 0.3 of \cite{neetc}.

\item\label{ilf}
 The  finite presentation of values of functors and the vanishing conditions can  certainly be treated separately.
So, we could have replaced the category $R-\mmodd$ of finitely presented $R$-modules by any other 
 exact subcategory of $R-\modd$ in our definitions and formulations. In particular, one can consider certain "levels of $\cu$-smallness" of objects of $\du$ (cf. Theorem \ref{tsmash}(II.2)). 





\item\label{iez} Certainly, one can take $R=\z$; this allows to apply parts I and III of our theorem to arbitrary triangulated categories. 
\end{enumerate}
\end{rema}

\begin{coro}\label{csatur} 
Assume 
that $\cu$ is $R$-saturated (where $R$ is a coherent ring) and $w$ is a bounded weight structure on it.
Then the following statements are valid.

1. For any $i\in \z$ and $M\in \obj \cu$ the functors $\tau_{\le i }(H_M)$  and   $\tau_{\ge  i }(H_M)$   are representable, where $H_M=\cu(-,M)$.

2. There exists a $t$-structure right adjacent to $w$. Moreover, $t$ is bounded if $\cu$ is densely generated by a single object $G$.

3. Its heart $\hrt$ is naturally  equivalent to $\adfur(\hw\opp,R-\mmodd)$. 
\end{coro}
\begin{proof}
Putting $\cu'=\du=\cu$ in Theorem \ref{tsatur}(II) we obtain everything except the boundedness of $t$ in assertion 2.

Now, assume that $\cu$ is densely generated by $\{G\}$. Since the functor $H_G=\cu(-,G)$ is finite, $\cu(G,G[j])=\ns$ for almost all $i\in \z$. Thus we can apply Proposition \ref{psingen}(2) for $\cu'=\cu$ to obtain that $t$ is bounded.
\end{proof}

\subsection{On coherent sheaf examples for the theorem}\label{sperf1}

Now assume that $R$ is a (commutative unital)  noetherian ring.

\begin{pr}\label{pnee1}
Let $X$ be a scheme that is proper over $\spe R$.

1. Take $\du=D_{qc}(X)$ (the unbounded derived category of quasi-coherent sheaves on $X$), $\cu=D^{perf}(X)$ (the triangulated category  of perfect complexes of coherent sheaves on $X$), and  $\cu'=D^b_{coh}(X)$ (resp. $\cu'=D_{coh}^-(X)$; here  $D_{coh}^b(X)$ is the bounded derived category of coherent sheaves on $X$, and $D_{coh}^-(X)$ is its bounded above version). 

Then $\cu\cong \cu \opp$, and the "finite versions" of those assumptions of Theorem \ref{tsatur}(I-III) that do not mention $w$  (resp. the "locally finite versions"  of the corresponding assumptions of Theorem \ref{tsatur}(I,III)) are fulfilled for these categories. 


2. Assume that $X$ is a regular scheme (that is proper over $\spe R$). Then the category $\cu=D^{perf}(X)$ equals $\cu'=D^b_{coh}(X)$; thus it is $R$-saturated.

Moreover, $\cu$ 
 is densely generated by a single object $G$. 

\end{pr}
\begin{proof}
1. Since $R$ is noetherian, $X$ is a noetherian scheme; thus it is well known that objects of $\cu$ compactly generate $\du$ and $\cu\cong \cu \opp$. Thus it remains to  apply Corollary 0.5 of \cite{neesat}; note here that the set of $R$-linear transformations between two $R$-linear functors between $R$-linear categories coincides with the set of transformations between the 
 underlying  additive functors.

2. It is well known that in this case $D^{perf}(X)=D^b_{coh}(X)$ indeed; thus $\cu$ is $R$-saturated. Moreover, 
  $\cu$ is densely generated by a single object according to Theorem 0.5 of \cite{neestrong} (see also Remark 0.6 of loc. cit.). 
\end{proof}

\begin{rema}\label{rtem1}
\begin{enumerate}
\item\label{item2}
Thus for $X$ and $\cu$ as in part 2 of the proposition both left and right adjacent $t$-structures to any bounded weight structure exist.

Moreover, any 
weight structure $w$ on $D^{perf}(X)$ (for $X$ that is proper over $\spe R$) gives a smashing $t$-structure on $D_{qc}(X)$ that is right orthogonal to $t$. This $t$ restricts to $D^-_{coh}(X)$ if $w$ is bounded below, and (also) restricts to $D^b_{coh}(X)$ if $w$ is bounded.

\item\label{igor} More generally, if $X$ is a Gorenstein (but not necessarily regular) scheme then the coherent duality functor $D_X$ gives an equivalence $D^{b}_{coh}(X)\opp\to D^b_{coh}(X)$ that restricts to an equivalence $D^{perf}(X)\opp\to D^{perf}(X)$. Hence for any bounded  weight structure $w$ on $D^{perf}(X)$ there also exists a left orthogonal weight structure on $D^b_{coh}(X)$.

\item\label{irecw}
Let us discuss the question which $t$-structures do possess left adjacent weight structures (in the setting of our proposition).  

Suppose that we start with a $t$-structure $t$ on $\cu=\cu'$. The easy Theorem \ref{twfromt}(I.1) below implies that there are enough projectives in $\hrt$ if there exists a weight structure $w$ that is left adjacent to $t$. Moreover, part IV of that theorem says that if this condition is fulfilled, $t$ is bounded, and $\cu\opp$ is $R$-saturated then $w$ does exist. 

Now we combine part 2 of our of our proposition with  Corollary \ref{csatur}(2) 
to obtain the following: if $X$ is a regular scheme that is proper over $\spe R$ and $\cu=D^{perf}(X)$ then there exists a 1-to-1 correspondence between bounded weights structures on $\cu$ and (their right adjacent) bounded $t$-structures such that $\hrt$ has enough projectives. Since $\cu$ is self-dual we also obtain a 1-to-1 correspondence between bounded weights structures on $\cu$ and bounded $t$-structures such that $\hrt$ has enough injectives. Certainly, these two statements can be combined to obtain a rather curious correspondence between those  bounded $t$-structures such that $\hrt$ has enough projectives and those  bounded $t$-structures such that $\hrt$ has enough injectives. Note also that there exists plenty of examples of bounded weight structures on $\cu$ whenever $X=\p^n(\spe R)$ (and $R$ is a regular ring); see 
 part \ref{item3} of this remark.

We will discuss a similar problem for orthogonal structures in Remark \ref{rtem2}(\ref{irecwo}) below.

\item\label{iconj}
The author also conjectures that all non-degenerate weight structures and $t$-structures on $D^{perf}(X)$ are bounded in this case. 

Note here that one can easily obtain non-trivial degenerate weight structures and $t$-structures $D^{perf}(X)$ by gluing (at least) if $X$ is a projective space (say, over a field); cf. part \ref{item3} this remark. Certainly, 
 degenerate weight and $t$-structures are not bounded. 


\item\label{item3} To make our proposition "practical" one needs to have some  bounded weight structures on a triangulated category $\cu$ as in our proposition.

Unfortunately, no "simple" general constructing methods similar to that provided by Theorem 5 of \cite{paucomp} (see Remark \ref{rsmashex}(1)) are available in this setting. However, one can {\it glue} weight structures (see Remark \ref{rsmashex}(2) and Theorem 8.2.3 of \cite{bws}). That is, if a triangulated category $\cu$ is glued from certain $\du$ and $\eu$ in the sense of \cite[\S1.4]{bbd} then any pair of (bounded) weight structures gives a "compatible" weight structure on $\eu$. Note here that one can shift weight structures on $\du$ and $\eu$ in the obvious way; thus a single pair of weight structures on $\du$ and $\eu$ gives a family of weight structures on $\cu$ indexed by $\z\times \z$ (whereas shifting a weight structure of this sort corresponds to adding $(i,i)$ to these parameters).

Now let us pass to more concrete examples; cf. also 
 Remark \ref{rsatur}(\ref{idercoh}). 
 One can construct a rich family of bounded weight structures on 
 $D^{perf}(X)$ at least in the case where $X=\p^n$ for some $n>0$ 
 (one can probably take any regular base ring $R$ here) by means of gluing. More generally, it suffices to assume that 
 $D^{perf}(X)$ possesses a {\it full exceptional collection} of objects.
Since one can "shift weight structures on components", one can obtain plenty of non-trivial weight structures and $t$-structures on $D^{perf}(X)$. 

Another important statement is that bounded weight structures $\cu$ are in one-to-one correspondence with those additive retraction-closed subcategories $\bu$ of $\cu$ such that $\bu$ strongly generates $\cu$ and $\bu$ is  negative in $\cu$; 
 see Proposition \ref{pcneg} below. 
 Thus one may look for negative subcategories in $D^{perf}(X)$ to obtain examples of weight structures.

\item\label{itemh} It appears that the first result in the direction of Proposition \ref{pnee1}(2) was Theorem 2.14 of \cite{bondkaprserr} where the case of a smooth projective variety over a field was considered. However, if $X$ is singular then one has to take $\cu'\neq \cu$; the corresponding statements were only recently established by Neeman, and they motivated our Definition \ref{drefl}  along with those results of this paper that depend on it. 

Recall also  that Theorem 4.3.4 of \cite{bvdb} 
 gives a certain a "non-commutative geometric" example of an $R$-saturated category (for $R$ being a field).

\item\label{item4} The author suspects that some of the statements in our theorem can be generalized to the case where $R$ is (not not necessarily noetherian itself but) coherent and can be presented as a "flat enough" direct limit of noetherian rings.
\end{enumerate}
\end{rema}

\subsection{Other possible examples related to coherent sheaves}\label{sperf2}

Once again, $R$ is a noetherian ring. We recall some more results of Neeman and D. Murfet.

\begin{pr}\label{pnee2}

Let $X$ be a scheme that is proper over $\spe R$; take $\cu= D^b_{coh}(X)\opp$.

1.  Assume in addition that {\it regular alterations} (see Remark \ref{rtem2}(\ref{item1}) below) exist for all 
 integral closed subschemes of $X$; take the categories $\du=D_{qc}(X)\opp$ 
 and $\cu'=D^{perf}(X)\opp$ (resp. $\cu'=D^-_{coh}(X)\opp$).

Then the "finite versions" of those assumptions of Theorem \ref{tsatur}(I,II) that do not mention $w$ 
  (resp. the "anti-locally finite versions"  of the corresponding assumptions of Theorem \ref{tsatur}(I)) are fulfilled for 
	our $(\cu,\cu',\du)$. 

2. Take $\du$ 
to be  
the mock homotopy category $\kmp X)$ of projectives over $X$ as defined in   \cite[Definition 3.3]{mur}; see Remark \ref{rtem2}(\ref{rmur}) below for more detail. Then $\du$ is compactly generated by (the image with respect to a full embedding of) $\cu$, and $\cu$ essentially equals the subcategory of compact objects of $\du$.

\end{pr}
\begin{proof}
1.  This is most of Theorem 0.2 of \cite{neetc}.

2. See Theorems 4.10 and 7.4 of \cite{mur}.
\end{proof}

\begin{rema}\label{rtem2}
\begin{enumerate}
\item\label{iunb} Now let us discuss the relation of these statements to the main subject of the paper.

The main problem is that if $X$ is not regular then it is well-known that there exist objects $M$ and $N$ of $\cu\opp=D^b_{coh}(X)$ such that $\cu\opp(M,N[i])\neq \ns$ for arbitrarily large values of $i$ (this is an easy consequence of the Serre criterion; cf. Proposition \ref{pnee1}(1) and Remark \ref{rsatur}(\ref{ireg})); hence there cannot exist any bounded weight structures on $\cu$. Thus one has no chance to apply the finite version of Theorem \ref{tsatur} in this setting.

The author does not know whether there can exist bounded 
 below weight structures on $\cu$ in this case. On the other hand, if $X=\spe R$ then it is easily seen that that the stupid weight structure on the category $K^-(\operatorname{Proj}R)$ of bounded above complexes of $R$-modules restricts to its subcategory $\cu\opp=D^b_{coh}(X)$; certainly, the corresponding weight structure on $\cu$ is bounded 
above. The author suspects that this example can be extended at least to the case $X=\p^n(\spe R)$ for any $n\ge 0$; cf.  Remark \ref{rtem1}(\ref{item3}).

\item\label{rmur} One can obtain a certain $t$-structure on a subcategory  of $\du=\kmp X)$ by means of Theorem \ref{tsatur}(III) if one takes the subcategory $\cu'$ of $\du$ corresponding to those anti-locally finite functors from $\cu=D^b_{coh}(X)\opp$ that are represented by objects of $\du$ and chooses a bounded 
above weight structure on $\cu$.

On the other hand, the embedding  $D^b_{coh}(X)\opp\to \kmp X)$ is given by a rather non-trivial construction from \S7 of \cite{mur}. So, the author does not currently know how to compute this category $\cu'$. Note however that the aforementioned example of a bounded below weight structure on $\cu$ (in the case $X=\spe R$) yields that our theory in non-vacuous in this setting.

\item\label{item1} So, Theorem \ref{satur} does not allow us to obtain any orthogonal $t$-structures from Proposition \ref{pnee2}(1). However, we will soon explain that that the latter proposition is useful for "recovering" left orthogonal weight structures. 

Thus it is worth 
 recalling that alterations were introduced in \cite{dej}; they generalize Hironaka's resolutions of singularities. Since the latter exist for arbitrary quasi-excellent $\spe\q$-schemes according to  Theorem 1.1 of \cite{temr}, part 1 of our proposition can be applies whenever $R$ is an quasi-excellent noetherian $\q$-algebra. Moreover, regular alterations of all integral (closed) subschemes of $X$ exist whenever $X$ is of finite type over a scheme $S$ that is quasi-excellent of dimension at most $3$; see Theorem 1.2.5 of \cite{tema}.

\item\label{irecwo} Similarly to Remark \ref{rtem1}(\ref{irecw}), let us now discuss to which extent the orthogonality relation between weight structures and certain $t$-structures in bijective in the setting of Proposition \ref{pnee1}(1). 

So, assume that $X$ is proper over $\spe R$. If  $w$ is a bounded weight structure on $\cu=D^{perf}(X)$ then there exists an orthogonal $t$-structure $t$ on  $\cu'= D^b_{coh}(X)$  such that $\hrt\cong \adfur(\hw\opp,R-\mmodd)$; see Remark \ref{rtem1}(\ref{item2}). 

Conversely, assume that $t$ is a bounded $t$-structure  on $\cu'$  such that $\hrt$ is equivalent to the category of $R$-linear functors from an $R$-linear category $\hu$ into $R-\mmodd$, and regular alterations exist for all 
 integral closed subschemes of $X$. Combining Proposition \ref{pnee2}(1) with Proposition \ref{psatw} below we obtain that  there exists a bounded weight structure $w''$ on a subcategory $\cu''$ of $\cu$ that is strictly 
	 orthogonal to $t$.
Moreover, if $t$ is actually orthogonal to a bounded weight structure $w$ on $\cu$ then we obtain $\cu''=\cu$ and $w''=w$ here; see Remark \ref{recw}.

 So the only obstacle for obtaining a one-to-one correspondence for these $\cu$ and $\cu'$ (and under our assumptions on $X$) is that we do not know whether $t$ that is orthogonal to a bounded weight structure on $\cu$ as above is necessarily bounded. Possibly, this question is related to Theorem 0.15 of \cite{neestrong}.
\end{enumerate}
\end{rema}

\subsection{On perfect complexes and coherent rings: a reminder}\label{scoher}

We will recall some basics on finitely presented modules and coherent rings.  Below we will only consider left $R$-modules, where $R$ is associative unital ring; 
 moreover, recall that in \S\ref{satur} we assume that $R$ is commutative.

\begin{defi}\label{dfp}
\begin{enumerate}
\item\label{idfp1}
We will say that a (left) $R$-module $M$ is {\it finitely presented} if there exists an exact sequence $P_1\to P_0\to M\to 0$ of $R$-modules, where $P_i$ are finitely generated $R$-projective.

\item\label{idfp2} We will say that $R$ is (left) {\it coherent} if any finitely generated left ideal of $R$ is finitely presented.

\item\label{idfp3} Moreover, $R$ is (left) {\it semi-hereditary} if any finitely generated left ideal of $R$ is projective.

\item\label{idfp4} We will use the notation $D(R)$ for the derived category of (left) $R$ modules. Moreover, we will write $D^{perf}(R)$ for the full 
  subcategory of $D(R)$ of {\it perfect complexes}, i.e. its objects are quasi-isomorphic to bounded complexes of finitely generated projective $R$-modules. 
\end{enumerate}
\end{defi}

Now we recall some basic properties of these notions.

\begin{pr}\label{pfp}

\begin{enumerate}
\item\label{ipfp1} Any valuation ring is semi-hereditary.

\item\label{ipfp2} All semi-hereditary and (left) noetherian rings are coherent.

\item\label{ipfp3} If $R$ is coherent then any finitely generated submodule of a finitely presented module is finitely presented, and finitely presented  modules form an exact abelian subcategory of $R-\modd$.


\item\label{ipfp5} If $R$ is semi-hereditary and $P$ is a finitely generated projective $R$-module then any finitely generated submodule of $P$ is projective.

\item\label{ipfp4} If $R$ is coherent  and a finitely presented $R$-module $M$ has an $R$-projective resolution of length $n$ then it also possesses a projective resolution of length $n$ whose terms are finitely presented $R$-modules.

\item\label{ipfp6} $D(R)$ is compactly generated by $\{R\}$ (considered as a left module over itself and put in degree $0$ as a complex), and $D(R)^{(\alz)}=D^{perf}(R)$. 
\end{enumerate}

\end{pr}
\begin{proof}
All of these statements appear to be rather well-known. Moreover, assertions \ref{ipfp1} and \ref{ipfp2} are obvious (note that any finitely generated projective $R$-module is finitely presented);  assertion \ref{ipfp3} is immediate from Theorem 2.4 of \cite{swancoh} along with Theorem A.9 of \cite{cohn} (where right modules over a ring $R$ were considered) and assertion \ref{ipfp5} is straightforward from  Corollary 0.3.3 of ibid. 

\ref{ipfp4}. We recall that if $0\to N\to P\to Q\to 0$ is an exact sequence of $R$-modules, $P$ is projective, and $Q$ is of projective dimension at most $j$ for some $j>0$ (i.e., $Q$ has a projective resolution of length $j$)  then $N$ is of projective dimension at most $j-1$. Hence the following easy inductive argument gives the statement: if $n=0$ then $M$ is projective and finitely presented itself; otherwise we can take an exact sequence $0\to N\to P\to M\to 0$ with $P$ being finitely presented projective to obtain that $N$ is of projective dimension at most $n-1$ and also finitely presented according to assertion \ref{ipfp3}.

\ref{ipfp6}. $D(R)$ is obviously compactly generated by $\{R\}$. Applying Lemma 4.4.5 of \cite{neebook} we obtain that $D(R)^{(\alz)}$ 
 equals $\lan \{R\}\ra$ (i.e.,  the subcategory of $\cu$ densely generated by $\{R\}$). Thus it remains to note that $D^{perf}(R)$ is a full strict triangulated subcategory of $D(R)$ that obviously lies in 
 $\lan \{R\}\ra$ (look at the distinguished triangles corresponding to stupid truncations of complexes), and to obtain the equality in question we apply the fact that $D^{perf}(R)$ is Karoubian (that is well-known and also follows immediately from Proposition \ref{pcneg}(2) below).  
\end{proof}


\section{Converse results: the existence of orthogonal weight structures}\label{sortw}

In this section we 
 study the question 
 when a $t$-structure  possesses a left adjacent or orthogonal weight structure. So, in certain cases we are able to recover $w$ from a right adjacent $t$-structure $t$ that can be constructed using the results of previous sections.

For this purpose in \S\ref{slemma} we recall some statements related to the construction of weight structures.

In \S\ref{sdual} we recall the (aforementioned) general definition of duality between two triangulated categories and construct an interesting family of examples.

In \S\ref{sconstrw} we prove that the existence of a left adjacent weight structure is closely related to the existence of enough projectives in the heart of $t$. We also prove the existence of a left orthogonal 
 weight structure in a context related to Theorem \ref{tsatur} and Proposition \ref{pnee1}(2). 

\subsection{Some lemmas 
 and the existence of weight structures: a reminder}\label{slemma}

\begin{pr}\label{pstar}
 Assume that $A$ and $B$ 
 are extension-closed classes of objects of $\cu$. 

 I. Assume that $A\perp B[1]$. Then the class $A\star B$  of all extensions of elements of $B$ by elements of $A$ is extension-closed as well.

II. Assume in addition that $\cu$ is smashing, and $A$ and $B$ are closed with respect to $\cu$-coproducts. 

1. Then  $A\star B$  is closed with respect to $\cu$-coproducts as well.

2. Assume that $A$ 
 is closed either with respect to $[-1]$ or with respect to $[1]$. Then $A$ 
 is retraction-closed in $\cu$.  


\end{pr}
\begin{proof}
All of these statements are rather easy.

Assertions I and II.1 
 immediately follow from Proposition 2.1.1 of \cite{bsnew}. Assertion II.2 is a straightforward consequence of Corollary 2.1.3(2) of 
 ibid. (see  Remark 2.1.4(4) of ibid.)
\end{proof}

We also recall a generalization of a well-known existence of weight structures result from \cite{bws}.

\begin{pr}\label{pcneg}
Let $\bu$ be an 
 additive negative subcategory (see Definition \ref{dwso}(\ref{id6})) 
of a triangulated category $\cu$ such that 
$\cu$ is densely generated by $\obj \bu$. 

 1. Then the envelopes  (see \S\ref{snotata}) $\cu_{w\le 0}$ and $\cu_{w\ge 0}$ of the classes $\cup_{i\le 0} \obj \bu[i]$ and $\cup_{i\ge 0} \obj \bu[i]$, respectively, give a  bounded weight structure $w$ on $\cu$, and 
 $\hw$ equals  $\kar_{\cu}(\bu)$. 

2.  $\obj \kar_{\cu}(\bu)$  strongly generates $\cu$. Moreover, $\cu_{w\le 0}$ (resp.  $\cu_{w\ge 0}$) is the extension-closure of $\cup_{i\le 0}\obj \kar_{\cu}(\bu)[i]$ (resp. of $\cup_{i\ge 0}\obj \kar_{\cu}(\bu)[i]$). 
\end{pr}
\begin{proof}
This is just (most of) Corollary 2.1.2 of \cite{bonspkar}.\end{proof}

\subsection{On dualities and (strict) orthogonality of structures}\label{sdual}

Let us now recall the definition of duality.

\begin{defi}\label{ddual}

Let $\au$ be an abelian category. 

 1. We will call a (covariant) bi-functor   $\Phi:\cu^{op}\times\cu'\to \au$ a {\it duality} if  it is bi-additive, homological with respect
to both arguments, and is equipped with a (bi)natural bi-additive transformation
$\Phi(-,-)\cong \Phi (-[1],-[1])$.


2. Suppose that  $\cu$ is endowed with a weight structure $w$,  $\cu'$ is endowed with a $t$-structure $t$. Then we will say that $w$
 is {\it left  orthogonal} to $t$ and $t$ is {\it right  orthogonal} to $w$ with respect to $\Phi$  if the following  { orthogonality condition} is fulfilled:
  $\Phi (X,Y)=0$ if $X\in \cu_{w\le 0}$ and $Y\in \cu'_{t \ge 1}$ or if $X\in \cu_{w\ge 0}$ and $Y\in \cu'_{t \le -1}$.

3. Assume that $t$ is right  orthogonal to $w$ with respect to $\Phi$. 

Then we will also say that $t$  is {\it $-$-orthogonal} (resp. {\it $+$-orthogonal}) to $w$ (with respect to $\Phi$) if
for any $Y\in \obj \cu'$ we have $Y\in \cu'_{t \le -1}$ (resp. $Y\in \cu'_{t \ge 1}$) whenever $\Phi (X,Y)=0$ for all $X\in \cu_{w\ge 0}$ (resp. $X\in \cu_{w\le 0}$).

Moreover, we will say that $t$  is {\it strictly right  orthogonal} to $w$ and $w$  is {\it strictly left  orthogonal} to $t$ if $t$ is both $-$ and $+$-orthogonal to $w$.



4. We will  write $P_t$ 
for the class of those $X\in \obj \cu$ such that $\Phi (X,Y)=0$ for all $Y\in \cu'_{t \le -1}\cup \cu'_{t \ge 1}$ (cf. \S3 of \cite{zvon}); $P'_t={}^{\perp_{\cu'}} (\cu'_{t\le -1}\cup \cu'_{t\ge 1})$.
\end{defi}

\begin{rema}\label{rdualzero}
1. If $\cu$ and $\cu'$ are  triangulated subcategories of a triangulated category $\du$ then the restriction of the bi-functor $\du(-,-)$ to $\cu\opp\times \cu'$ is obviously a duality. Thus Definition \ref{ddual} is compatible with Definition \ref{dort}. 

More generally, if $i:\cu\to\du$ and $i':\cu'\to \du$ are arbitrary exact functors then $\du(i(-),i'(-))$ is a duality as well. Moreover, this duality is {\it nice} in the sense of Definition 2.5.1(2) of \cite{bger}. However, the notion of niceness is not relevant for the purposes of the current paper. 

2. Certainly, $\Phi=0$ is a duality, and any $w$ and $t$ on the corresponding categories are orthogonal with respect to it. Thus a certain strictness condition appears to be quite actual (at least) in 
 the setting of general dualities.

The importance of this notion is also illustrated by Proposition \ref{phw}(II.4) below (note that it is applied in Theorem \ref{twfromt}(I.3)).\end{rema}

Now let us study the relation between $\hw$ and $\hrt$.

\begin{pr}\label{phw}
Assume that   $t$  is  right  orthogonal to $w$  with respect to $\Phi$.

I.1. Then $\cu_{w=0}\subset P_t$.   Moreover, for any object $M\in P_t$ the functor $\Phi(M,-)$ restricts to an exact functor $E^M:\hrt\to \au$, and we have $\Phi(M,-)\cong E^M\circ H_0^t$.

2. Assume in addition that $t$  is  $-$ or $+$-orthogonal to $w$  (with respect to $\Phi$). Then 
$\{E^M:\ M\in \cu_{w=0}\}$ is a conservative collection of functors $\hrt\to \au$ (cf. Remark \ref{rcons} below).

3. Conversely, assume that functors of the type $E^M$ for $M\in \cu_{w=0}$ form a conservative collection and $t$ is right (resp. left) non-degenerate. Then $t$  is  $-$ (resp. $+$) right  orthogonal to $w$.

II. Assume 
that $\cu\subset \cu'$ and $\Phi$ is the restriction of $\cu'(-,-)$ to $\cu\opp\times \cu'$, i.e., $t$  is  
 right  orthogonal to $w$ in $\cu'$. 

1. Then 
 $\cu_{w\ge 0}= \cu'_{t\ge 0}\cap \obj \cu$,  $\cu_{w\le 0}= \perpp \cu'_{t\ge 1}\cap \obj \cu$, and $\cu_{w=0}=P_t$. 

2. For any $P\in P_t$ we have  natural isomorphisms of functors 
$$\cu'(P,-)\cong \cu'(P,H_0^t(-))\cong \hrt (H_0^t(P), H_0^t(-));$$ 
the first of them is induced by the transformations $ \id_{\cu'}\to t_{\le 0}$ and $H_0^t\to t_{\le 0}$ (see Remark \ref{rtst}(1)). 

 3. Assume that $t$ is cosmashing, $\cu$ equals $\cu'$ and satisfies the dual Brown representability property. Then $H_0^t$ gives an equivalence of (the full subcategory of $\cu$ given by) $P_t$ with the subcategory of projective objects of $\hrt$.

4. If $t$ is $+$-orthogonal to $w$ then $\cu'_{t\ge 0}$ is closed with respect to $\cu'$-products.
\end{pr}
\begin{proof}
I.1. $\cu_{w=0}\subset P_t$ immediately from our definitions.

If $0\to A_1\to A_2\to A_3\to 0$ is a short exact sequence in $\hrt$ then $A_1\to A_2\to A_3\to A_1[1]$ is well-known to be a distinguished triangle. Applying the functor  $\Phi(M,-)$ to it and recalling the definition of 
 $P_t$ we obtain an exact sequence  $0=\Phi(M,A_3[-1])\to E^M(A_1)\to E^M(A_2)\to E^M(A_3)\to \Phi(M,A_1[1])=0$; hence $E^M$ is exact indeed. Moreover, the functors $\Phi(M,-)$ and $ E^M\circ H_0^t$  are homological and annihilate both $\cu'_{t\le -1}$ and $\cu'_{t\ge 1}$; hence they are isomorphic.

2. Since all of these functors are exact, it suffices to verify that for any non-zero $N\in \cu'_{t=0}$ there exists $M\in \cu_{w=0}$ such that $\Phi(M,N)\neq 0$. Now, if $\Phi(M,N)=0$ for all $M\in \cu_{w=0}$ then Theorem 2.1.2(2) of \cite{bwcp} easily implies that $\Phi(M',N)=0$ for any $M'\in \obj \cu$. Combining this statement with either $-$ or $+$-orthogonality of $t$ to $w$ we immediately obtain $N=0$ (i.e., a contradiction).

3. If $t$ is right (resp. left)  non-degenerate, it suffices  to verify that $H_i^t(N)=0$ whenever $i<0$ (resp. $i>0$), $N\in \obj \cu'$,   and $\Phi(M,N)=0$ for all $M\in \cu_{w=i}$. For this purpose it is certainly sufficient to check that $H_0^t(N)=0$ whenever
$\Phi(M,N)=0$ for all $M\in \cu_{w=0}$. The latter statement is immediate from our assumption on $\cu_{w=0}$ along with the isomorphism $\Phi(M,-)\cong E^M\circ H_0^t$ provided by assertion I.1.

II.1. 
The argument is rather similar to the proof of Proposition \ref{prefl} (cf. Remark \ref{rhop}). We will use the notation $(C_1,C_2)$ for the couple $(\perpp\cu'_{t\ge 1}\cap \obj \cu,\ \cu'_{t\ge 0}\cap \obj \cu)$.

Since $w$ is left orthogonal to $t$, we have $\cu_{w\le 0}\subset  C_1$ and $\cu_{w\ge 0}\subset  C_2$. Next, $C_1\perp C_2[1]$, and applying Proposition \ref{pbw}(\ref{iort}) we also obtain inverse inclusions. Thus $w=(C_1,C_2)$ indeed; hence $\cu_{w=0}=C_1\cap C_2=P_t$. 

2. All of these 
statements are rather easy; 
 they are given by Lemma 2(1) of \cite{zvon}.

3. We should prove that any projective object $P_0$ of $\hrt$ "lifts" to $P_t$. 

Now, the functor $H_0^t$ respects products according to the easy Lemma 1.4 of \cite{neesat}  (applied in the dual form; cf. also Proposition 3.4(2) of \cite{bvt}); thus the composition $G^{P_0}=\hrt(P_0,-)\circ H_0^t:\cu\to \ab$ respects products as well. Moreover, $G^{P_0}$ is obviously homological functor. Thus it is corepresentable by some $P\in \obj \cu$ that certainly belongs to $P_t$, and it remains to apply the previous assertion.

4. Obvious.
\end{proof}

\begin{rema}\label{rcons}
 Since the functors of the type $E^M$ that we consider in part I of our proposition are exact (on $\hrt$), the conservativity of $\{E^M:\ M\in \cu_{w=0}\}$ is 
 fulfilled if and only if for any non-zero $N\in \cu'_{t=0}$ there exists $M\in \cu_{w=0}$ such that $E^M(N)\neq 0$.
\end{rema}

To  describe an interesting family of orthogonal structures (for $\hrt$ that is not necessarily semi-simple) we need the following definition that is closely related to Definition  D.1.13 of \cite{bvk} (see Remark \ref{rwfilab}(1) below).  

\begin{defi}\label{dwfilab} 
For an abelian category $\au$ we will say that an increasing family of full strict 
abelian subcategories $\au_{\ge i}\subset \au$, $i\in \z$, 
give a {\it nice split filtration}  for $\au$ if $\cup_{i\in \z}\obj\au_{\ge i}=\obj \au$, $\cap_{i\in \z}\obj\au_{\ge i}=\ns$, 
 and there exist exact left adjoints $W_{\ge i}$ to the embeddings  $\au_{\le i}\subset \au$.
\end{defi}

\begin{pr}\label{portss}
Adopt the assumptions of Definition \ref{dwfilab}.

1. Then $\au_{\ge i}$  are actually Serre subcategories of $\au$.

Moreover, the localization functors $\au_{\ge i}/\au_{\ge i+1}\to \au_i $ possess exact left adjoints $l_i$. So we will assume that $\au_i $ are subcategories of $\au$; for an object $N$ of $\au$ and any $j\in \z$ we will write $N_j$ for $l_j(W_{\ge i}M)$.

Furthermore, $\obj\au_i\perp \obj\au_j$ whenever $i\neq j$.

2. Assume that all the categories $\au_i$ are (abelian) semi-simple. Then the category $\au'$ of semi-simple objects of $\au$ consists of finite coproducts of objects of various $\au_i$. 

3. Let $\cu'$ be a triangulated category endowed with a non-degenerate (see Definition \ref{dtstro}(3)) $t$-structure $t$ such that $\hrt=\au$; take $\cu=K^b(\au')$.

Then the pairing $\Phi:\cu\opp\times \cu'\to \ab$ that sends $(M,N)$ into $\bigoplus_{i\in \z}(\bigoplus_{j\in \z} H_i(M)_j, H_{i}^t(N)_j)$, where $H_*(M)$ is the homology of the complex $M$ and  $H_{*}^t(N)$ is the $t$-homology of $N$, is a duality. Moreover, $t$  is  strictly right  orthogonal to  the stupid weight structure $w$ on $\cu$ with respect to $\Phi$.
\end{pr}
\begin{proof}
1. The exactness of $W_{\ge i}$ easily implies that $\au_{\ge i}$ are Serre subcategories of $\au$  indeed (cf. Remark D.1.20 of \cite{bvk}). Hence the existence and exactness of $l_i$ is  given by Lemma D.1.15 of ibid. (applied in the dual form; see  Remark \ref{rwfilab}(1)).

2. We should prove that simple objects of $\au$ are precisely the elements of $\cup \obj \au_i$.  This is immediate from Proposition D.1.16(1) of ibid.

3. The choice of the isomorphism $\Phi(-,-)\cong \Phi (-[1],-[1])$ is obvious. Next, $\cu$ is semi-simple; hence any its object 
 is a sum of shifts of objects of $\au_i$. Combining this fact with the exactness of all $l_j\circ W_{\ge j}:N\mapsto N_j$ along with the semi-simplicity of $\au_j$ we obtain that the functor $\Phi(M,-)$ is homological for any $M\in \obj \cu$. Moreover, the semi-simplicity of $\cu$ means that any distinguished triangle in it is a sum  of rotations of triangles of the form $M\to M\to 0\to M[1]$; this observation easily implies that the functor $\Phi(-,N)$ is cohomological for any object $N$ of $\cu'$.

Furthermore,  $t$  is  obviously  right  orthogonal to   $w$ with respect to $\Phi$. Since $t$ is non-degenerate, to verify strictness it suffices to check that for any  non-zero $N\in \cu'_{t=0}$ there exists $M\in \cu_{w=0}$ such that $\Phi(M,N)\neq 0$. The latter statement is an easy consequence of the definition of $\Phi$; see Proposition D.1.16(3) of loc. cit. 
\end{proof}

\begin{rema}\label{rwfilab}

1. The difference of our definition \ref{dwfilab} from the definition of a {\it weight filtration} in D.1.14 of ibid. is that we reverse the arrows and change the sign of inequalities (for $\au_{\le i}$).

2. Assume that endomorphism rings of simple objects of $\au$ are commutative (and so, they are fields). Then one can easily "replace" the duality $\Phi$ in Proposition \ref{portss}(3) by a duality $\Phi':\cu\opp\times \cu\to \au'$ such that $\Phi'(M,N)=M$ whenever $N=M$ are simple objects of $\au'\subset \au$ and $\Phi'(M,N)=0$ if $M$ and $N$ are non-isomorphic simple objects of $\au'$.

3. Certainly, we could have taken $\cu=K(\au')$ instead of $\cu=K^b(\au')$ in Proposition \ref{portss}(3); one can also take any intermediate triangulated category here.
\end{rema}

\subsection{Existence of 
 orthogonal weight structures on subcategories}\label{sconstrw}


Now we study the question which weight structures are adjacent to weight structures; yet in certain cases we are only able to construct a weight structure on a subcategory of the corresponding category.


\begin{theo}\label{twfromt}
Let $t$ be a $t$-structure on $\cu'$.

I. Assume that there exists a weight structure $w$ on a triangulated category $\cu\subset \cu'$ that is left orthogonal to $w$ (in $\cu'$; here we use Definition \ref{dort}(1)) and  $\cu'_{t=0}\subset \obj \cu$.

1. Then there are enough projectives in $\hrt$, 
 for any $M\in  \cu'_{t=0}$ there exists an $\hrt$-epimorphism from 
 $H_0^t(P)$ into $ M$ for some $P\in \cu_{w=0}$,   and the functor $H_0^t$  induces  an equivalence of $\kar(\hw)$
  with the category of projective objects of $\hrt$.  

2. Moreover, $\hw$ is equivalent to the latter category whenever the class $\obj\cu$ is retraction-closed in $\cu'$ and $\cu'$ is Karoubian.

3. Assume in addition that $t$ is left non-degenerate. Then $t$ is  $+$-orthogonal to $w$; hence $\cu'_{t\le 0}$ is closed with respect to $\cu'$-products. 

II. Assume that there are enough projectives in $\hrt$ and for any projective object $P'$ of $\hrt$ there exists $P\in P'_t$ (see Definition \ref{ddual}(4)) along with an  
$\hrt$-epimorphism $H_0^t(P)\to P'$.

1. Then the full subcategory $\cu$ of $\cu'$ whose object class equals $\cup_{i\in \z}\cu'_{t\ge i}$ is triangulated, and 
there exists a weight structure $w$ on $\cu$  
such that  $\cu_{w\ge 0}=\cu'_{t\ge 0}$ and $t$ is $-$-orthogonal to $w$. 

 2. 
Furthermore, one can 
 extend  (see Definition \ref{dwso}(\ref{idrest})) $w$ as above to a weight structure $w'$ on $\cu'$ that  is left adjacent to $t$ 
 whenever any of the following additional assumptions is fulfilled:

a.  $t$ is bounded below (see Definition \ref{dtstro}(4)).  
 
b. There exists 
an integer $n$ such that $\cu_{t\le 0}\perp \cu_{t\ge n}$. 
 
III. Assume in addition that $\cu'$ satisfies the dual Brown representability property (see Definition \ref{dcomp}(\ref{idbrown})), $t$ is cosmashing and $\hrt$ has enough projectives. Then the category $\cu'$ is smashing,  there exists a weight structure $\wu$ on the localizing subcategory $\ccu$ of $\cu'$ that is generated by $\cu'_{t\le 0}$ such that  $\cu_{\wu\le 0}=\cu'_{t\le 0}$, and $\hwu$ is equivalent to the subcategory of  projective objects of $\hrt$.

IV. Assume that $R$ is a commutative unital coherent ring, the category $\cu'{}\opp$ is $R$-saturated, $t$ is bounded, and $\hrt$ has enough projectives. Then there exists a weight structure $w'$ on $\cu'$ that is left adjacent to $t$. 
\end{theo}
\begin{proof}

I.1. 
Fix $M\in \cu_{t=0}$ and consider its 
$w$-decomposition $P\stackrel{p}{\to} M\to w_{\ge 1}M\to P[1]$. Since $M\in \cu'_{t= 0}$, Proposition \ref{phw}(II.1) implies that $P$ belongs to $\cu_{w\ge 0}$; hence $P $ belongs to $ \cu_{w=0}$ according to Proposition \ref{pbw}(\ref{iwd0})). Next,  $P\in  \cu_{w\ge 0}\subset \cu'_{t\ge 0}$ (see Proposition \ref{phw}(II.1)); hence  the object  $P_0=t_{\le 0}P$  equals  $H_0^t(P)$ (see Remark \ref{rtst}). Therefore $P_0$ is projective in $\hrt$ according to Proposition \ref{phw}(II.2).

Next, the adjunction property for the functor $t_{\le 0}$ (see Remark \ref{rtst}(1)) implies that $p$ factors through the $t$-decomposition 
morphism $P\to P_0$.  Now we check that the corresponding morphism $P_0\to M$ is an $\hrt$-epimorphism. This is certainly equivalent to its cone $C$ belonging to $\cu_{t\ge 1}$. The octahedral axiom of triangulated categories gives a distinguished triangle $(t_{\ge 1}P)[1]\to w_{\ge 1}M\to C\to (t_{\ge 1}P)[2]$; it yields the assertion in question since $w_{\ge 1}M\in \cu_{w\ge 1}\subset \cu'_{t\ge 1}$ 
 and the class $\cu_{t\ge 1}$ is extension-closed. Thus we obtain that $\hrt$ has enough projectives.

Now, the category of projective objects of $\hrt$ is certainly Karoubian. As we have just verified,   for any projective object $Q$ of $\hrt$ there exists an $\hrt$-epimorphism $H_0^t(S)\to Q$ for some $S\in \cu_{w=0}$. Since $H_0^t(S)$ is projective in $\hrt$ according  to Proposition \ref{phw}(II.2), this epimorphism splits, i.e.,   $Q$ equals the image of some idempotent endomorphism of $H_0^t(S)$. Applying Proposition \ref{phw}(II.2) once again and lifting this endomorphism to $\hw$ we obtain that $\kar(\hw)$ is equivalent to the category of projective objects of $\hrt$ indeed. 

2. 
 Since $\cu_{w=0}$ is  retraction-closed in $\cu'$, $\hw$ is Karoubian as well. Hence $\hw\cong \kar(\hw)$ in this case and we obtain the result in question.

3. If $t$ is  $+$-orthogonal to $w$ then $\cu'_{t\le 0}$ is closed with respect to $\cu'$-products according to Proposition \ref{phw}(II.4). 

Applying Proposition \ref{phw}(I.3) we obtain that it remains to verify that the functors of the form $E^M$ for $M\in \cu_{w=0}$ give a conservative family of functors $\hrt\to \ab$. Now, for any object $N$ of $\hrt$ our assumptions give the existence of a projective object $P_0$ of $\hrt$ that surjects onto it. Moreover, applying Proposition \ref{phw}(II.2) we obtain the existence of $P\in \cu_{w=0}$ and a morphism $h$ from $P$ 
  such that $H_0^t(h)$ is isomorphic to this surjection $P_0\to N$. Hence $E^P(N)\neq 0$ if $N$ is non-zero, and we obtain the conservativity in question (see Remark \ref{rcons}).

II.1.  
$\cu$ is  triangulated since the functor $H_0^t:\cu'\to \hrt$ is homological. Next we take $C_1=\perpp \cu'_{t\ge 1}\cap \obj \cu$, $C_2=\cu_{t\ge 0}$, and prove that $(C_1,C_2)$ is a weight structure on $\cu$ (cf. Proposition \ref{phw}(II.1)).

The only non-trivial axiom check here is the existence of $w$-decompositions for all objects of $\cu$. 
 Let us verify the existence of a $w$-decomposition for any $M\in \cu_{t\ge i}$ by induction on $i$. The statement is obvious for $i> 0$ since $M\in \cu_{w\ge 1}=\cu_{t\ge 1}$ and we can take a "trivial" weight decomposition $0\to M\to M\to 0$. 

Now assume that existence of $w$-decompositions is known for any $M\in \cu_{t\ge j}$ for some $j\in \z$. We should verify the existence of weight decomposition of an element $N$ of $\cu_{t\ge j-1}$. Certainly, $N$ is an extension of $N'[-j-1]=H_0^t(N[j+1])[-j-1]$   by $t_{\ge j}N$ (see Remark \ref{rtst}(1) for the notation). Since the latter object possesses a weight decomposition, Proposition \ref{pstar}(I) (with $A=C_1$ and $B=C_2[1]$) allows us to verify the existence of a weight decomposition of $N'[-j-1]$ (instead of $N$). Our assumptions imply that there exists an epimorphism $H_0^t(P)\to N'$ with $P\in P'_t=C_1\cap C_2$. 
 Then a cone $C $ of the corresponding composed morphism  $P\to N'$ is easily seen to belong to $\cu'_{t\ge 1}$. Since both $P$ and $C$ possess weight decompositions, applying Proposition \ref{pstar}(I) once again we obtain the statement in questions.

 Lastly, $t$ is $-$-orthogonal to $w$ immediately from Remark \ref{rtst}(3).

2. If the assumption a is fulfilled then we can just take $w'=w$ since  $\cu$ obviously equals $\cu'$.

Now suppose that assumption b is fulfilled. Similarly to the previous proof, 
  it suffices to verify that for the couple 
$w'=
 (\perpp \cu'_{t\ge 1},\cu'_{t\ge 0})$ the corresponding $w'$-decompositions exist for all objects of $\cu'$.

 Since $w'$ is an extension of $w$, assertion II.1 gives the existence of $w'$-decompositions for all elements of $\cu'_{t\ge 2-n}$. Next, our  orthogonality assumption on $t$ yields that $\cu'_{t\le 1-n}\subset \cu_{w\le 0}$; hence one can take trivial 
 $w'$-decompositions for elements of  $\cu'_{t\le 1-n}$. It remains to note that $\obj \cu'=\cu'_{t\ge 2-n}\star \cu'_{t\le 1-n}=\obj \cu'$ by axiom (iv) of weight structures, and apply Proposition \ref{pstar}(I) once again.

III. $\cu'$ is smashing according to  Proposition 8.4.6 of  \cite{neebook} (applied in the dual form). Applying Proposition \ref{phw}(II.2--3) we obtain that $\cu'$ and $t$ satisfy the assumptions of assertion II.1. We argue similarly to its proof and verify that  the corresponding $(\tilde{C}_1,\tilde{C}_2)$ give a weight structure $\wu$ on $\ccu$. Once again, for this purpose it suffices to verify that the class $\tilde{C}=\tilde{C}_1\star \tilde{C}_2[1]$ equals $\obj \ccu$. Immediately from assertion II, $\tilde{C}$ contains $\cu'_{t\ge j}$ for all $j\in \z$. Moreover, $\tilde{C}$ is extension-closed and closed with respect to $\cu'$-coproducts according to Proposition \ref{pstar}(I, II.1); hence $\tilde{C}$ equals $\obj \ccu$ indeed.

Lastly, $\obj\ccu$ is retraction-closed in $\cu'$ and $\cu'$ is Karoubian according to  Proposition \ref{pstar}(II.2); hence $\hwu$ is
 equivalent to the subcategory of  projective objects of $\hrt$ according to assertion I.2.

IV. We want to apply assertion II.1 in this case; to check its assumptions it certainly suffices to verify that for any projective object $P'$ of $\hrt$ there exists $P\in P'_t$ 
such $P'\cong H_0^t(P)$. According to Proposition \ref{phw}(II.2), the latter statement is equivalent to the corepresentability of the functor $G^{P_0}=\hrt(P_0,-)\circ H_0^t:\cu\to R-\modd$ (cf. the proof of Proposition \ref{phw}(II.3)). Now, the values of $G^{P_0}$ are finitely presented $R$-modules (since morphisms between any two objects of $\cu$; here we also apply Proposition \ref{pfp}(\ref{ipfp3})). Since $t$ is bounded, $G^{P_0}$ is finite in the sense of Definition \ref{dsatur}(1), and we obtain the corepresentability in question (see Definition \ref{dsatur}(2)).

Lastly, the corresponding category $\cu$ equals $\cu'$ since $t$ is bounded (cf. assertion II.2; one can also apply its formulation directly). Since the resulting weight structure $w'=w$ is left orthogonal to $t$, it is also left adjacent to it (see Definition \ref{dort}(3)). 
\end{proof}

\begin{rema}\label{rtfindim}
1. Certainly, part III of our theorem becomes more interesting in the case $\cu=\cu'$.

2. Moreover, parts I and III of our theorem can be considered as a certain complement to Theorem \ref{tsmash}(I). So we obtain that the class of 
 $t$-structures right adjacent to smashing weight ones is "closely related" to the one of cosmashing $t$-structures such that $\hrt$ has enough projectives.

3. Similarly, parts I, II, and IV of our theorem complement Corollary \ref{csatur}(2,3); see Remark \ref{rtem1}(\ref{irecw}). 

It is also an interesting question whether for a general $R$-saturated  category $\cu$ a weight structure $w$ that is left adjacent to a $t$-structure $t$ is bounded if and only if $t$ is.

4. The condition $\cu_{t\le 0}\perp \cu_{t\ge n}$ for $n\gg 0$ (see part II.2 of our theorem) is a natural generalization of the finiteness of the cohomological dimension condition (for an abelian category).

5. In parts II and III of our theorem the "starting points" for  constructing weight decompositions were the classes $\cu'_{t\ge 0}$ and $P'_t$. Now, 
 it is easily seen that one can construct a weight structure starting from $P_t\subset \obj \cu$ if $\cu\opp$ reflects $\cu'{}\opp$ (in a certain $\du\opp$ or with respect to a duality) using Proposition \ref{pcneg}. Certainly, one will also need certain corepresentability assumptions to prove that $P_t$ is large enough in some sense (if it is).
\end{rema}

\begin{pr}\label{psatw}

Assume that $R$ is a commutative unital coherent ring, $\cu$ and $\cu'$ are $R$-linear triangulated subcategories of an $R$-linear triangulated category $\du$,  
  the corresponding functor $M\mapsto H^M_{\cu'}$ gives an equivalence of $\cu\opp$ with the category of finite homological functors from $\cu'$ into $R-\modd$ (i.e., the functors of the type $H^M$ are finite as $R$-linear functors from $\cu'{}\opp$ into $R-\modd$), and $t$ is a bounded $t$-structure on $\cu'$ such that $\hrt$ is equivalent to the category of $R$-linear functors from an $R$-linear category $\hu$ into $R-\mmodd$. 

Then there exists a bounded weight structure $w''$ on the subcategory $\cu''=\lan P_t\ra$ of $\cu$ such that $\hw''= P_t\supset \hu$ and 
 $w''$ is strictly left orthogonal to $t$.
\end{pr}
\begin{proof}
First we prove that $P_t$ gives a negative subcategory of $\cu$. So we fix $A,B\in P_t$, $i\ge 0$, and prove that $A\perp B[i]$ by an argument somewhat similar to the proof of Proposition \ref{pwrange}(\ref{iwruni}).

So $T$ be a transformation $H^{B[i]}_{\cu'}\to H^{A}_{\cu'}$. We can obviously complete it to a commutative square 
$$\begin{CD} 
H^{B[i]}_{\cu'}\circ t_{\ge i}@>{}>>H^{A}_{\cu'}\circ t_{\ge i} \\
@VV{T'}V@VV{}V \\
H^{B[i]}_{\cu'}@>{T}>>H^{A}_{\cu'}\end{CD} $$ 
Since the transformation $T'$ is an isomorphism and $H^{A}_{\cu'}\circ t_{\ge i}=0$, we obtain $T=0$.

Applying Proposition \ref{pcneg} we obtain the existence of a bounded weight structure $w''$ on the category $\cu''=\lan P_t\ra$ such that $\cu''_{w''=0}=\kar_{\cu}P_t=P_t$. Next, the description of  $\cu''_{w''\le 0}$ and $\cu''_{w''\ge 0}$ provided by this proposition easily implies that $w''$ is left orthogonal to $t$. 

Now, $t$ is non-degenerate since it is bounded. Applying Proposition \ref{phw}(I.3) we obtain that to finish the proof and verify strong orthogonality it suffices to verify that any object $M$ of $\hu$ lifts to an element of $P_t$. Now  
 we argue similarly to the proof of Proposition \ref{phw}(II.3). The functor $G^{N}=\hrt(N,-)\circ H_0^t:\cu\to R-\modd$ is homological; it takes values in finitely presented $R$-modules immediately from Proposition \ref{pfp}(\ref{ipfp3}).  Since $t$ is bounded, it is finite as a functor from $\cu'{}\opp$ into $R-\modd$, and we obtain that it is corepresentable by an object of $\cu$.
\end{proof}

\begin{rema}\label{recw}
Now assume that the $t$-structure $t$ in our proposition is orthogonal to a weight structure $w$ on the whole category $\cu$; cf. Remark \ref{rtem1}(\ref{irecwo}). Then $P_t$ contains $\cu_{w=0}$; hence $\cu''=\cu$ (see Proposition \ref{pcneg}(1)). Moreover,  Propositions \ref{pbw}(\ref{iuni}) and \ref{pcneg}(1) easily imply that $w''=w$. 

However, one does not have $\cu''=\cu$ in general. In particular, one can easily construct an example with $\cu\neq 0$ and $\cu'=0$; in this case we certainly have $\cu''=0\neq \cu$.
\end{rema}

\end{document}